\documentclass[12pt]{amsart}
\usepackage{tikz}

\newtheorem{thm}{Theorem}
\newtheorem{prp}{Proposition}
\newtheorem{lem}{Lemma}
\newtheorem{cor}{Corollary}
\newtheorem{conj}{Conjecture}
\newtheorem{obs}{Observation}

\theoremstyle{definition}
\newtheorem{rem}{Remark}

\newtheorem{example}{Example}
\DeclareMathOperator{\Des}{Des}
\DeclareMathOperator{\Asc}{Asc}
\DeclareMathOperator{\des}{des}
\DeclareMathOperator{\col}{col}

\newcommand{\el}[2]{\genfrac{\langle}{\rangle}{0pt}{}{#1}{#2}}

\title{A two-sided analogue of the Coxeter complex}

\author[T. K. Petersen]{T. Kyle Petersen}
\thanks{The author was supported by a Simons Foundation collaboration grant.}
\address{Department of Mathematical Sciences, DePaul University, Chicago, IL, USA}
\email{tpeter21@depaul.edu}
\urladdr{http://www.math.depaul.edu/\~{}tpeter21}


\begin{document}

\maketitle

\begin{abstract}
For any Coxeter system $(W,S)$ of rank $n$, we introduce an abstract boolean complex (simplicial poset) of dimension $2n-1$ that contains the Coxeter complex as a relative subcomplex. Faces are indexed by triples $(I,w,J)$, where $I$ and $J$ are subsets of the set $S$ of simple generators, and $w$ is a minimal length representative for the parabolic double coset $W_I w W_J$. There is exactly one maximal face for each element of the group $W$. The complex is shellable and thin, which implies the complex is a sphere for the finite Coxeter groups. In this case, a natural refinement of the $h$-polynomial is given by the ``two-sided" $W$-Eulerian polynomial, i.e., the generating function for the joint distribution of left and right descents in $W$. 
\end{abstract}

\section{Introduction}

Coxeter groups were developed to study symmetries of regular polytopes, and they play a major role in the study of Lie algebras (the Weyl group of a root system is a Coxeter group). The Coxeter complex is a simplicial complex associated with the reflection representation of the group, but which can also be defined abstractly via cosets of parabolic subgroups. The goal of this note is to provide a ``two-sided" analogue of the Coxeter complex by considering double cosets of parabolic subgroups. 

Before turning to the new construction, let us recall some definitions and important properties of the usual Coxeter complex. We assume the reader has some familiarity with the study of Coxeter groups. See Humphreys' book \cite{Humphreys} or Bj\"orner and Brenti's book \cite{BjB} for background.

Fix a finitely generated Coxeter system $(W,S)$, and let $W_J$ denote the standard parabolic subgroup generated by a subset of simple generators $J \subseteq S$. It is well known that the set of cosets of parabolic subgroups forms an abstract simplicial complex known as the \emph{Coxeter complex}, and denoted by
\[
 \Sigma = \Sigma(W,S) = \{ wW_J : w \in W, J\subseteq S\}.
\]
The faces of $\Sigma$ are ordered by reverse inclusion of cosets, i.e.,
\[
 wW_J \leq_{\Sigma} w'W_{J'} \quad \mbox{ if and only if } \quad  wW_J \supseteq w'W_{J'}.
\]
Note that this means maximal elements are singleton sets: $w W_{\emptyset} = \{w\}$, and there is a unique minimal element: $w W_S = W$. Some well-known features of the Coxeter complex are highlighted in the following result. Most of these statements can be found in work of Bj\"orner \cite{Bjorner2} (see also Abramenko and Brown \cite[Chapter 3]{AbramenkoBrown}), though some of these facts were known earlier. See, e.g., Bourbaki \cite{Bourbaki}.

\begin{thm}\label{thm:Coxeter}
For any Coxeter system $(W,S)$ with $|S|=n<\infty$ we have the following.
\begin{enumerate}
\item The Coxeter complex $\Sigma$ is a balanced simplicial complex of dimension $n-1$.
\item The facets (maximal faces) of $\Sigma$ are in bijection with the elements of $W$.
\item The Coxeter complex is shellable and any linear extension of the weak order on $W$ gives a shelling order for $\Sigma$.
\item If $W$ is infinite then $\Sigma$ is contractible.
\item If $W$ is finite,
\begin{enumerate}
\item the geometric realization of $\Sigma$ is a sphere, and
\item the $h$-polynomial of $\Sigma$ is the $W$-Eulerian polynomial,
\[
 h(\Sigma;t) = \sum_{w \in W} t^{\des(w)},
\]
where $\des(w)$ denotes the number of descents of the element $w$.
\end{enumerate}
\end{enumerate}
\end{thm}

We will try to emulate all these properties for a ``two-sided" version of the Coxeter complex, denoted $\Xi = \Xi(W,S)$. While we will defer the definition of $\Xi$ to Section \ref{sec:defn}, let us comment on one matter. Although its faces are related to parabolic double cosets $W_IwW_J$, $\Xi$ is \emph{not} merely the set of such cosets ordered by inclusion. (See Remark \ref{rem:distinct}.)
Our main results are summarized as follows.

\begin{thm}\label{thm:main}
For any Coxeter system $(W,S)$ with $|S|=n<\infty$, we have the following.\begin{enumerate}
\item The complex $\Xi$ is a balanced boolean complex of dimension $2n-1$.
\item The facets (maximal faces) of $\Xi$ are in bijection with the elements of $W$, and the Coxeter complex $\Sigma$ is a relative subcomplex of $\Xi$.
\item The complex $\Xi$ is shellable and any linear extension of the two-sided weak order on $W$ gives a shelling order for $\Xi$.
\item If $W$ is infinite then $\Xi$ is contractible.
\item If $W$ is finite,
\begin{enumerate}
\item the geometric realization of $\Xi$ is a sphere, and
\item a refined $h$-polynomial of $\Xi$ is the two-sided $W$-Eulerian polynomial,
\[
 h(\Xi;s,t) = \sum_{w \in W} s^{\des_L(w)}t^{\des_R(w)},
\]
where $\des_L(w)$ denotes the number of left descents of $w$ and $\des_R(w)$ denotes the number of right descents of the element $w$.
\end{enumerate}
\end{enumerate}
\end{thm}

The main contrasts between $\Xi$ and $\Sigma$ lie in the fact that $\Xi$ is roughly twice the dimension of $\Sigma$ and in the fact that $\Xi$ is not a simplicial complex. While all the faces of $\Xi$ are simplices, many of these simplices share the same vertex set. Even for the rank one Coxeter group $A_1$, $\Xi$ is realized by two edges whose endpoints are paired off to form a circle: \begin{tikzpicture} \draw (0,0) circle (5pt); \draw (-.2,0) node[circle,fill=black,inner sep=1] {}; \draw (.2,0) node[circle,fill=black,inner sep=1] {};\end{tikzpicture}.

We remark that our approach in this work is combinatorial, not geometric. There are two different approaches to proving the topological results for the Coxeter complex listed in Theorem \ref{thm:Coxeter}. One way (following Bourbaki \cite{Bourbaki}) is to study the reflection hyperplane arrangement for the Coxeter group. For example, in the finite case, intersecting this arrangement with a sphere realizes the Coxeter complex. 
%For the affine groups, we realize the complex with affine hyperplanes in Euclidean space; for hyperbolic groups, a hyperbolic space. 
Thus in this situation the topology of the Coxeter complex is manifest in the ambient space. On the other hand, Bj\"orner showed in  \cite{Bjorner2} how to use poset-theoretic tools to study the topology of the complex with only the abstract definition of the face poset. 

The approach of this paper mirrors that of Bj\"orner. We define the face poset of $\Xi$ abstractly, and use Bj\"orner's techniques to deduce Theorem \ref{thm:main}. We hope to uncover a more geometric description of $\Xi$ in the future.

The paper is structured as follows. 

The first few sections introduce $\Xi$ and establish the various parts of our main theorem. In Section \ref{sec:defn} we provide the definition of $\Xi$ and the proof of parts (1) and (2) of Theorem \ref{thm:main}. In Section \ref{sec:topology} we prove parts (3), (4), and (5a) of Theorem \ref{thm:main}. Section \ref{sec:faces} discusses face enumeration in the case of finite groups $W$, and establishes part (5b) of Theorem \ref{thm:main}.

In Section \ref{sec:poly} we define, for any finite Coxeter group $W$, the ``two-sided" Eulerian polynomials
\[
 W(s,t) := \sum_{w \in W} s^{\des_L(w)} t^{\des_R(w)}.
\]
These polynomials have pleasant properties and we offer a generalization of a conjecture of Gessel that asserts that these polynomials expand positively in the basis
\[
 \{ (st)^a(s+t)^b(1+st)^{n-2a-b} \}_{0\leq 2a+b \leq n},
\]
where $n$ is the rank of the group. See Conjecture \ref{conj:gamma}.

Finally in Section \ref{sec:contingency} we discuss a combinatorial model for faces of $\Xi$ in the case of the Coxeter group of type $A_{n-1}$, i.e., the symmetric group $S_n$. Here the faces of $\Xi$ can be encoded by two-way \emph{contingency tables}. These tables are nonnegative integer arrays whose entries sum to $n$ and whose row and column sums are positive. The partial order on faces in this case is simply refinement ordering on contingency tables. Maximal tables are permutation matrices and the minimal element is the unique one-by-one array. Such arrays were studied by Diaconis and Gangolli \cite{DG}, but not this partial ordering on the arrays.

The author would like to thank John Stembridge for helpful conversations that inspired the idea for this paper, and Vic Reiner for comments on an earlier draft.

\section{A two-sided Coxeter complex}\label{sec:defn}

Throughout this section we assume familiarity with basic Coxeter group concepts and terminology. We mostly follow the definitions and notational conventions of \cite{BjB} and \cite{Humphreys}.

Fix a Coxeter system $(W,S)$ with $|S|=n$. We call the elements $s \in S$ the \emph{simple generators} of $W$. 
%Let $T = \{ w s w^{-1} : w \in W, s \in S\}$ denote the set of \emph{reflections} in $W$. 
Every element $w  \in W$ can be written as a product of elements in $S$, $w=s_1\cdots s_k$, and if this expression is minimal, we say the \emph{length} of $w$ is $k$, denoted $\ell(w) = k$. An expression of minimal length is called a \emph{reduced expression}.

Recall that a \emph{cover relation} in a partially ordered set (``poset" for short) is a pair $x < y$ such that if $x\leq z \leq y$, then $x=z$ or $z=y$. A partial ordering of a set can be defined as the transitive closure of its cover relations. One important partial order on $W$ is known as the \emph{weak order}. The weak order comes in two equivalent types: ``left" and ``right" weak order. We will also have reason to mention the ordering obtained from the union of the covers in left weak order and right right weak order, which we call the ``two-sided" weak order. We now describe these orderings in terms of their cover relations. 
\begin{itemize} 
\item The \emph{left weak order} on $W$ says $v$ covers $u$ if and only if $\ell(v) = \ell(u)+1$ and $u^{-1}v \in S$. 
\item The \emph{right weak order} on $W$ says $v$ covers $u$ if and only if $\ell(v) = \ell(u)+1$ and $vu^{-1} \in S$.
\item The \emph{two-sided weak order} on $W$ says $v$ covers $u$ if and only if $\ell(v) = \ell(u)+1$ and $vu^{-1}$ or $u^{-1}v$ is in $S$.
\end{itemize}

The left and the right weak order are obviously subposets of the two-sided weak order. We write $u \leq_L v$ if $u$ is below $v$ in the left weak order, we write $u\leq_R v$ if $u$ is below $v$ in the right weak order, and we write $u \leq_{LR} v$ if $u$ is below $v$ in the two-sided weak order. The identity is the unique minimum in these partial orderings. When $W$ is finite, there is also a unique maximal element denoted $w_0$, and each poset is self-dual, i.e., isomorphic to its reverse ordering.

Though will do not use the fact, we mention that all three of these posets are subposets of the strong Bruhat order on $W$, whose covers have $u^{-1}v$ or $vu^{-1}$ equal to a conjugate of an element of $S$. 

The \emph{left (resp. right) descent set} of an element $w$ is the set of all simple generators that take us down in left (resp. right) weak order when multiplied on the left (resp. right). We denote the left and right descent sets by $\Des_L(w)$ and $\Des_R(w)$, respectively, i.e.,
\[
 \Des_L(w) = \{ s \in S : \ell(sw) < \ell(w) \} \mbox{ and }  \Des_R(w) = \{ s \in S : \ell(ws) < \ell(w) \}.
\]
We define the corresponding \emph{ascent sets} as the complements of the descent sets in $S$:
\[
 \Asc_L(w) = S-\Des_L(w)=\{ s \in S : \ell(sw) > \ell(w) \},
\]
and
\[
 \Asc_R(w) = S- \Des_R(w) = \{ s \in S : \ell(ws) > \ell(w)\}.
\]

Intuitively, we move up and down in left (resp. right) weak order by multiplying elements on the left (resp. right) by simple generators. We move up and down in the two-sided weak order by multiplying on either side by simple generators. For Bruhat order, we move up and down by inserting simple generators anywhere in a given reduced expression. 

Suppose $J$ is a subset of simple generators, $J \subseteq S$, and let $W_J$ denote the group generated by the elements of $J$, i.e., $W_J = \langle s : s \in J\rangle$. This group is a Coxeter group in its own right, and we call such a subgroup a \emph{standard parabolic subgroup}. The Coxeter complex arises when considering the quotients of the form $W/W_J$. That is, the faces of the Coxeter complex are identified with left cosets of parabolic subgroups, $wW_J$. To be precise, let
\[
\Sigma = \bigcup_{J \subseteq S} W/W_J = \{ wW_J : w \in W, J \subseteq S \}.
\]
We partially order the elements of $\Sigma$ by reverse containment of sets, i.e., by declaring
\[
 wW_J \leq_{\Sigma} w'W_{J'},
\]
if and only if
\[
 wW_J \supseteq w'W_{J'}.
\]
The dimension of a face $wW_J$ is given by $\dim(wW_J) = |S-J|-1$, so that vertices correspond to cosets of the form $wW_{S-\{s\}}$, and maximal faces are singleton cosets of the form $wW_{\emptyset} = \{w\}$.

For our two-sided analogue, we consider elements from all double quotients $W_I\backslash W /W_J$, so the faces will be related to double cosets of parabolic subgroups $W_I w W_J$, where $I$ and $J$ are subsets of $S$. However, the faces of $\Xi$ are \emph{not} simply the double cosets of this form. See Remark \ref{rem:distinct}.

An essential fact about cosets of parabolic subgroups is that each coset $wW_J$ has a unique element of minimal length, call it $u$, such that $J \subseteq \Asc_R(u)$, or $\Des_R(u) \subseteq S-J$. In fact, the same is true for double cosets, and we record this in the following lemma, which can be found in \cite[Chapter 4, Exercise 1.3]{Bourbaki}.

\begin{lem}\label{lem:minrep}
Each double coset $W_I w W_J$ has a unique element of minimal length, call it $u$, such that
\[
 \Des_L(u) \subseteq S - I \quad \mbox{ and } \quad \Des_R(u) \subseteq S-J,
\]
or
\[
 I \subseteq \Asc_L(u) \quad \mbox{ and } \quad J \subseteq \Asc_R(u).
\]
Moreover, for each $v \in W_I w W_J$, $u$ is below $v$ in the two-sided weak order: $u \leq_{LR} v$. 
\end{lem}

Let ${}^I W^J$ denote the set of minimal representatives for $W_I\backslash W/W_J$, i.e.,
\[
{}^I W^J =\{ w \in W : I \subseteq \Asc_L(w) \mbox{ and } J \subseteq \Asc_R(w) \}.
\]
If $I = \emptyset$ we have ${}^{\emptyset} W^J = W^J$ is the set of left coset representatives.
 
With the lemma in mind, we could just as easily replace the cosets $wW_J$ in the definition of $\Sigma$ with pairs $(w,J)$ such that $w \in W^J$, i.e.,
\[\Sigma \cong \{ (w,J) : J \subseteq S, w \in W^J\}.\]
Extending this idea, we make the following definition.

Let
\[
\Xi = \{ (I, w, J) : I, J \subseteq S \mbox{ and } w \in {}^I W^J \}.
\]
 We partially order the elements of $\Xi$ by reverse inclusion of the index sets $I$ and $J$ as well as the corresponding double coset, i.e., 
\[
 (I, w, J) \leq_{\Xi} (I',w',J') \quad \mbox{ if and only if } \quad \begin{cases}
 I \supseteq I',\\
 J \supseteq J', \mbox{ and}\\
 W_IwW_J \supseteq W_{I'}wW_{J'}.
\end{cases}
\]
We will refer to the $\Xi$ as the \emph{two-sided Coxeter complex}.

\begin{example}
In Figure \ref{fig:A2} we see the poset of faces of the two-sided Coxeter complex $\Xi(A_2)$. Faces are written as triples $(I,w,J)$, where $I, J \subseteq \{s_1,s_2\}$. We write only the subscripts for brevity, e.g., $(\{s_1\},e,\{s_1,s_2\})$ is written $(1,e,12)$.
\end{example}

\begin{figure}
{\small
\[
 \begin{tikzpicture}[xscale=.6,yscale=2.5]
 \tikzstyle{state}=[fill=white,inner sep = 2,scale=.6];
 \tikzstyle{linet}=[line width=5, cap=round, color=white!80!black];
  \draw (0,0) node[state] (12e12) {$(12,e,12)$};
  \draw (-6,1) node[state] (1e12) {$(1,e,12)$};
  \draw (-2,1) node[state] (2e12) {$(2,e,12)$};
  \draw (6,1) node[state] (12e2) {$(12,e,2)$};
  \draw (2,1) node[state] (12e1) {$(12,e,1)$};
  \draw (-9,2) node[state] (e12) {$(\emptyset,e,12)$};
  \draw (-7,2) node[state] (1e2) {$(1,e,2)$};
  \draw (-5,2) node[state] (1e1) {$(1,e,1)$};
  \draw (-3,2) node[state] (2e2) {$(2,e,2)$};
  \draw (-1,2) node[state] (2e1) {$(2,e,1)$};
  \draw (1,2) node[state] (12e) {$(12,e,\emptyset)$};
  \draw (3,2) node[state] (2s12) {$(2,s_1,2)$};
  \draw (5,2) node[state] (1s21) {$(1,s_2,1)$};
  \draw (7.25,2) node[state] (2s1s21) {$(2,s_1s_2,1)$};
  \draw (9.5,2) node[state] (1s2s12) {$(1,s_2s_1,2)$};
  \draw (-11,3) node[state] (e2) {$(\emptyset,e,2)$};
  \draw (-9,3) node[state] (e1) {$(\emptyset,e,1)$};
  \draw (-7,3) node[state] (1e) {$(1,e,\emptyset)$};
  \draw (-5,3) node[state] (2e) {$(2,e,\emptyset)$};
  \draw (-3,3) node[state] (s12) {$(\emptyset,s_1,2)$};
  \draw (-1,3) node[state] (2s1) {$(2,s_1,\emptyset)$};
  \draw (1,3) node[state] (s21) {$(\emptyset,s_2,1)$};
  \draw (3,3) node[state] (1s2) {$(1,s_2,\emptyset)$};
  \draw (5.25,3) node[state] (s1s21) {$(\emptyset,s_1s_2,1)$};
  \draw (7.5,3) node[state] (2s1s2) {$(2,s_1s_2,\emptyset)$};
  \draw (9.75,3) node[state] (s2s12) {$(\emptyset,s_2s_1,2)$};
  \draw (12,3) node[state] (1s2s1) {$(1,s_2s_1,\emptyset)$};
  \draw (-10,4) node[state] (e) {$(\emptyset,e,\emptyset)$};
  \draw (-6,4) node[state] (s1) {$(\emptyset,s_1,\emptyset)$};
  \draw (-2,4) node[state] (s2) {$(\emptyset,s_2,\emptyset)$};
  \draw (2,4) node[state] (s1s2) {$(\emptyset,s_1s_2,\emptyset)$};
  \draw (6,4) node[state] (s2s1) {$(\emptyset,s_2s_1,\emptyset)$};
  \draw (10,4) node[state] (s1s2s1) {$(\emptyset,s_1s_2s_1,\emptyset)$};
  \draw[linet] (e)--(e2);
  \draw[linet] (e)--(e1);
  \draw[linet] (e)--(1e);
  \draw[linet] (e)--(2e);
  \draw[linet] (e2)--(e12);
  \draw[linet] (e2)--(1e2);
  \draw[linet] (e2)--(2e2);
  \draw[linet] (e1)--(e12);
  \draw[linet] (e1)--(1e1);
  \draw[linet] (e1)--(2e1);
  \draw[linet] (1e)--(1e2);
  \draw[linet] (1e)--(1e1);
  \draw[linet] (1e)--(12e);
  \draw[linet] (2e)--(2e2);
  \draw[linet] (2e)--(2e1);
  \draw[linet] (2e)--(12e);
  \draw[linet] (e12)--(1e12);
  \draw[linet] (e12)--(2e12);
  \draw[linet] (1e2)--(1e12);
  \draw[linet] (1e2)--(12e2);
  \draw[linet] (1e1)--(1e12);
  \draw[linet] (1e1)--(12e1);
  \draw[linet] (2e2)--(2e12);
  \draw[linet] (2e2)--(12e2);
  \draw[linet] (2e1)--(2e12);
  \draw[linet] (2e1)--(12e1);
  \draw[linet] (12e)--(12e2);
  \draw[linet] (12e)--(12e1);
  \draw[linet] (12e12)--(1e12);
  \draw[linet] (12e12)--(2e12);
  \draw[linet] (12e12)--(12e2);
  \draw[linet] (12e12)--(12e1);
  \draw (e)--(e2);
  \draw (e)--(e1);
  \draw (e)--(1e);
  \draw (e)--(2e);
  \draw (e2)--(e12);
  \draw (e2)--(1e2);
  \draw (e2)--(2e2);
  \draw (e1)--(e12);
  \draw (e1)--(1e1);
  \draw (e1)--(2e1);
  \draw (1e)--(1e2);
  \draw (1e)--(1e1);
  \draw (1e)--(12e);
  \draw (2e)--(2e2);
  \draw (2e)--(2e1);
  \draw (2e)--(12e);
  \draw (e12)--(1e12);
  \draw (e12)--(2e12);
  \draw (1e2)--(1e12);
  \draw (1e2)--(12e2);
  \draw (1e1)--(1e12);
  \draw (1e1)--(12e1);
  \draw (2e2)--(2e12);
  \draw (2e2)--(12e2);
  \draw (2e1)--(2e12);
  \draw (2e1)--(12e1);
  \draw (12e)--(12e2);
  \draw (12e)--(12e1);
  \draw (12e12)--(1e12);
  \draw (12e12)--(2e12);
  \draw (12e12)--(12e2);
  \draw (12e12)--(12e1);
  %Second cell
  \draw[linet] (s1)--(s12);
  \draw[linet] (s1)--(2s1);
  \draw[linet] (s12)--(2s12);
  \draw[linet] (2s1)--(2s12);
  \draw (s1)--(e1);
  \draw (s1)--(1e);
  \draw (s1)--(s12);
  \draw (s1)--(2s1);
  \draw (s12)--(e12);
  \draw (s12)--(1e2);
  \draw (s12)--(2s12);
  \draw (2s1)--(2e1);
  \draw (2s1)--(12e);
  \draw (2s1)--(2s12);
  \draw (2s12)--(2e12);
  \draw (2s12)--(12e2);
  %Third cell
  \draw[linet] (s2)--(s21);
  \draw[linet] (s2)--(1s2);
  \draw[linet] (1s2)--(1s21);
  \draw[linet] (s21)--(1s21);
  \draw (s2)--(e2);
  \draw (s2)--(2e);
  \draw (s2)--(s21);
  \draw (s2)--(1s2);
  \draw (s21)--(e12);
  \draw (s21)--(2e1);
  \draw (s21)--(1s21);
  \draw (1s2)--(1e2);
  \draw (1s2)--(12e);
  \draw (1s2)--(1s21);
  \draw (1s21)--(1e12);
  \draw (1s21)--(12e1);
  %fourth cell
  \draw[linet] (s1s2)--(s1s21);
  \draw[linet] (s1s2)--(2s1s2);
  \draw[linet] (2s1s2)--(2s1s21);
  \draw[linet] (s1s21)--(2s1s21);
  \draw (s1s2)--(s12);
  \draw (s1s2)--(1s2);
  \draw (s1s2)--(s1s21);
  \draw (s1s2)--(2s1s2);
  \draw (s1s21)--(e12);
  \draw (s1s21)--(1s21);
  \draw (s1s21)--(2s1s21);
  \draw (2s1s2)--(12e);
  \draw (2s1s2)--(2s12);
  \draw (2s1s2)--(2s1s21);
  \draw (2s1s21)--(2e12);
  \draw (2s1s21)--(12e1);
  %Fifth cell
  \draw[linet] (s2s1)--(s2s12);
  \draw[linet] (s2s1)--(1s2s1);
  \draw[linet] (s2s12)--(1s2s12);
  \draw[linet] (1s2s1)--(1s2s12);
  \draw (s2s1)--(s21);
  \draw (s2s1)--(2s1);
  \draw (s2s1)--(s2s12);
  \draw (s2s1)--(1s2s1);
  \draw (s2s12)--(e12);
  \draw (s2s12)--(2s12);
  \draw (s2s12)--(1s2s12);
  \draw (1s2s1)--(12e);
  \draw (1s2s1)--(1s21);
  \draw (1s2s1)--(1s2s12);
  \draw (1s2s12)--(1e12);
  \draw (1s2s12)--(12e2);
  %Sixth cell
  \draw (s1s2s1)--(s1s21);
  \draw (s1s2s1)--(2s1s2);
  \draw (s1s2s1)--(s2s12);
  \draw (s1s2s1)--(1s2s1);
\end{tikzpicture}
\]
}
\caption{The two-sided Coxeter complex $\Xi(A_2)$. Highlighted edges indicate the shelling.}\label{fig:A2}
\end{figure}
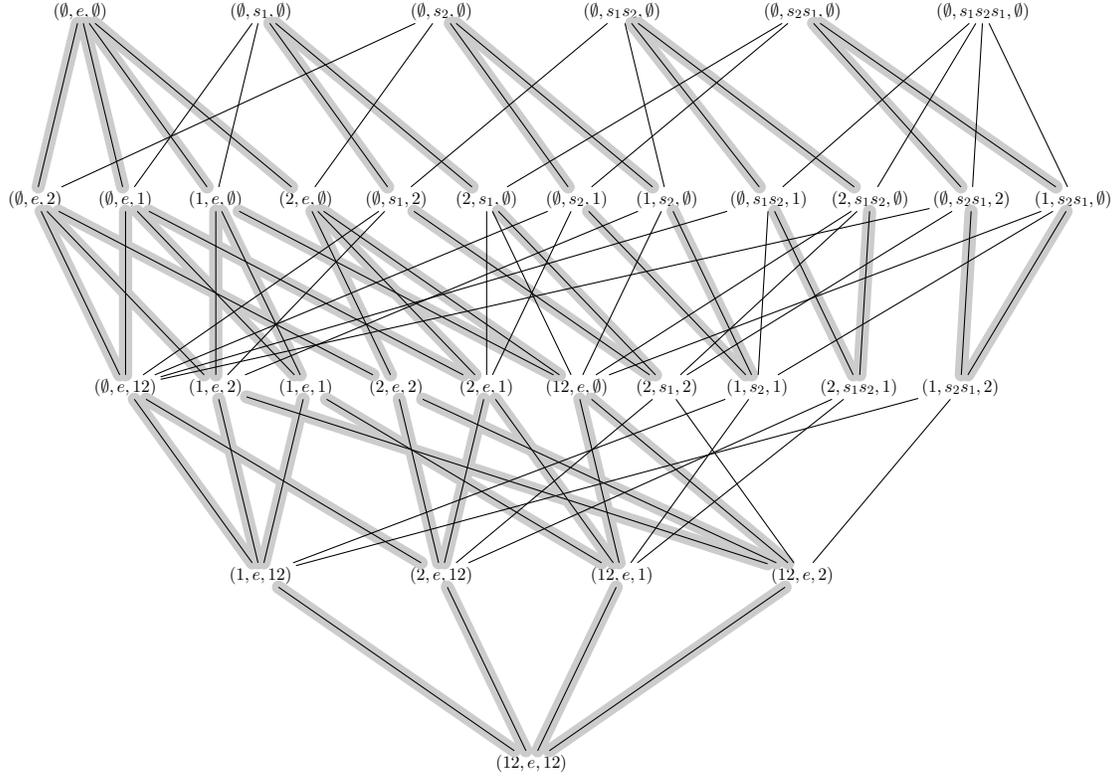

Before we move on to prove the various properties of $\Xi$ given in Theorem \ref{thm:main}, we include some remarks.

\begin{rem}\label{rem:distinct}
A first guess to define a two-sided Coxeter complex is to consider the set of all double cosets $W_I wW_J$, ordered by reverse inclusion. Such a poset does indeed exist, but it is difficult to analyze. It is not even obvious when this poset is ranked. For one thing, there are many subtle equalities of cosets, e.g., with $w$ fixed, we might have $W_I w W_J = W_{I'}wW_{J'}$ and yet $I\neq I'$ or $J \neq J'$. For an extreme case, notice that for any $I \subseteq J$, we have $W_I e W_J = W_J$. Enumeration of the number of distinct parabolic double cosets is the topic of work of Billey, Konvalinka, Petersen, Slofstra, and Tenner \cite{BKPST}.
\end{rem}

\begin{rem}\label{rem:DGS}
If we fix a choice of $I$ and $J$, we can restrict the Bruhat order on $W$ to give a partial ordering on the elements ${}^IW^J$, or on the double quotient $W_I\backslash W /W_J$. Stembridge gives a geometric construction of this partial order in terms of root systems \cite{Stembridge}. Diaconis and Gangolli did the same in the case of the symmetric group, realized as a partial order on contingency tables with prescribed row and column sums \cite{DG}.
\end{rem}

\subsection{$\Xi$ is boolean}

The maximal elements in $\Xi$ are those of the form $(\emptyset, w, \emptyset)$, and there is a unique minimum, $(S, e, S)$. The rank one elements are those of the form $(S-\{i\}, e, S)$ and $(S,e,S-\{j\})$, i.e., those obtained by omitting a single element from $S$ on either the left or on the right.

We will now prove that lower intervals in the poset $\Xi$ are isomorphic to boolean algebras. Since the face poset of a simplex is the boolean algebra on its vertex set, a poset with this property is known as a \emph{simplicial poset}, or as a \emph{boolean complex}.

\begin{thm}\label{thm:boolean}
The poset $\Xi$ is a simplicial poset. In particular, the interval below the element $(I,w,J) \in \Xi$ isomorphic to the set of all subsets of  $(S-I)\times (S-J)$.
\end{thm}

\begin{proof}
Fix an element $F=(I,w,J)$ of $\Xi$ and consider any element below $F$ in the partial order, i.e., suppose we have an element $(I',w',J') \leq_{\Xi} F$. Then by definition, $S \supseteq I' \supseteq I$ and $S \supseteq J' \supseteq J$, so $(I'-I, J'-J)$ is an element of $(S-I)\times (S-J)$. 

To finish the proof we must show that every pair of subsets $(I'-I, J'-J)$ in $(S-I)\times (S-J)$ corresponds to a unique element below $F$. 

Suppose $(I'-I, J'-J)$ is a pair of subsets in $(S-I)\times (S-J)$, i.e., $S \supseteq I' \supseteq I$ and $S \supseteq J' \supseteq J$. If $C = W_{I'} v W_{J'}$ is a coset that contains $W_I w W_J$, then in particular $w \in C$ and we can write $C=W_{I'} w W_{J'}$. Thus for fixed $I'$ and $J'$, there is one such coset. By Lemma \ref{lem:minrep} there exists a unique element $w' \in C$ such that $\Des_L(w') \subseteq S-I'$ and $\Des_R(w') \subseteq S-J'$. This identifies the unique triple $G=(I',w',J')$ such that $G \leq_{\Xi} F$, completing the proof. 
\end{proof}

Theorem \ref{thm:boolean} means that each element of $\Xi$ can be thought of as an abstract simplex. As such, we will refer to the elements as \emph{faces}. We say a face $(I,w,J)$ is \emph{represented by $w$}.

\subsection{$\Xi$ is balanced}

While each face of $\Xi$ is a simplex, it is not a simplicial complex, since distinct faces may share the same vertex set. In fact, we will see that for any $(W,S)$, $\Xi$ has the property that every \emph{facet} (maximal face) has the same vertex set.

The dimension of a face is given by one less than its rank in the poset, i.e., if $F = (I,w,J)$,
\[
 \dim F = |S-I|+|S-J| - 1.
\]
In particular, if $|S| = n$, then $\Xi$ has $2n$ vertices, each of the form $(S-\{i\},e,S)$ or $(S,e,S-\{j\})$. The dimension of $\Xi$ is the dimension of a maximal face, i.e., $\dim \Xi = \dim (\emptyset,w,\emptyset) = 2n-1$.

A $(n-1)$-dimensional simplicial complex is \emph{balanced} if there is an assignment of colors from the set $\{1,2,\ldots,n\}$ to its vertices such that each face has distinctly colored vertices. An important feature of the Coxeter complex $\Sigma$ is that it is balanced; a balanced coloring is given by declaring the \emph{color} of the pair $F=(w,J)$ is $\col(F) = S-J$, where we fix an identification between $S$ and the set $\{1,2,\ldots,n\}$.

To show $\Xi$ is balanced we will assign each vertex a color via
\[
 \col( (S-\{i\},e,S)) = (\{i\}, \emptyset) \mbox{ and } \col( (S,e,S-\{j\})) = (\emptyset, \{j\}), 
\]
and for a general face $F$, $\col(F)$ is the union of the colors of its vertices, i.e., if $F=(I,w,J)$, then
\[
 \col(F) = (S-I,S-J).
\]

Since there are $2n$ colors and only $2n$ vertices, we see that $\Xi$ is trivially balanced, i.e., no face has two vertices of the same color since every vertex has  a unique color. We have now established part (1) of Theorem \ref{thm:main}.

\subsection{$\Sigma$ is a relative subcomplex of $\Xi$}

We have already mentioned that maximal faces of $\Xi$ are in bijection with elements of $W$. Let us denote the facet corresponding to an element $w$ by $F_w=(\emptyset, w, \emptyset)$. If we consider only adding elements to the right we get a subposet of $\Xi$ that corresponds to a facet of the usual Coxeter complex. That is, consider the interval
\[
 [ (\emptyset,e,S), F_w ] = \{ G \in \Xi : (\emptyset,e,S) \leq_{\Xi} G \leq_{\Xi} (\emptyset, w, \emptyset) \}.
 \]
We can represent elements $G \in [ (\emptyset,e,S), F_w ]$ as $G= (\emptyset,u,J)$, such that $J \subseteq S$, $u \in {}^{\emptyset}W^J$, and $w \in uW_J$. 

Similarly, a facet of $\Sigma$ can be represented as an interval
\begin{align*}
 [(e,S), (w,\emptyset)] &= \{ G \in \Sigma : (e,S) \leq_{\Sigma} G \leq_{\Sigma} (w,\emptyset \},\\
  &= \{ (u,J) : J \subseteq S, u \in W^J, w \in uW_J\}.
\end{align*}

Thus as posets
\[
 [(\emptyset,e,S),F_w] \cong [ (e,S), (w,\emptyset)] \in \Sigma.
\]
(Of course the same idea would work with right cosets, so we could also identify facets of $\Sigma$ with intervals of the form $[(S,e,\emptyset),F_w]$ if we wish.) 

Taking the union of all such intervals we get a full copy of $\Sigma$ as an upper order ideal (also known as an order filter) inside of $\Xi$.
\begin{align*}
 \Sigma = \{ (w,J) : J \subseteq S, w \in W^J\} &\cong  \{ (\emptyset, w, J) : J \subseteq S, w \in {}^{\emptyset}W^J \},\\ 
 &=\{ F \in \Xi : (\emptyset, e, S) \leq_{\Xi} F \}.
\end{align*}
To phrase this result another way, we say that $\Sigma$ is a \emph{relative subcomplex} of $\Xi$. This establishes part (2) of Theorem \ref{thm:main}.

\subsection{$\Xi$ is partitionable}

We can notice that the faces represented by a given element $w$ form an upper interval in $\Xi$, i.e., they form an interval whose maximal element has maximal rank in the face poset. To be specific, let $R_w = (\Asc_L(w), w, \Asc_R(w))$, which we call the \emph{restriction} of $w$.  Then the interval $[R_w, F_w]$ in $\Xi$ consists of all faces represented by $w$, and moreover this interval is boolean:
\begin{align*}
 [R_w, F_w] &= \{ (I,w,J) : I \subseteq \Asc_L(w),  J\subseteq \Asc_R(w) \},\\
 & \cong  \Asc_L(w)\times \Asc_R(w). 
\end{align*}

The union of all such intervals partitions the faces of $\Xi$, i.e.,
\begin{equation}\label{eq:partition}
 \Xi = \bigcup_{w \in W} [R_w, F_w],
\end{equation}
and this union is disjoint. Moreover, since each interval in the partition is an upper ideal isomorphic to a boolean algebra, $\Xi$ is \emph{partitionable} in the topological sense as well. This property foreshadows the shellability result of the next section. See \cite[Section III.2]{Stanley} for the relevant definitions.

\section{Topology}\label{sec:topology}

In this section we will prove parts (3) and (4) of Theorem \ref{thm:main}. 

\subsection{$\Xi$ is shellable}\label{sec:shelling}

We first make the following simple observation. If $(I,u,J)$ is a face of $\Xi$ below the face $(I',v,J')$, then in particular $W_{I'}vW_{J'} \subseteq W_IuW_J$, and $v \in W_I uW_J$. But by Lemma \ref{lem:minrep} this means $u$ is below $v$ in the two-sided weak order.

\begin{obs}\label{obs:weak}
If $(I,u,J) \leq_{\Xi} (I',v,J')$, then $u\leq_{LR} v$.
\end{obs}

From this simple observation it follows that any choice of linear extension of the two-sided weak order for $W$ is a shelling order for $\Xi$. First recall the definition of a \emph{shelling} of a boolean complex. This is an ordering of the facets $F_1, F_2,\ldots$ such that the intersection of the boundary of each new facet with the union of the boundaries of the prior facets is a pure codimension one complex. That is, for each $k$, we must show
\[
 \partial F_k \cap \left( \bigcup_{i = 1}^{k-1} \partial F_i \right)
\]
is a pure codimension one complex. Here $\partial F_k$ denotes the boundary of $F_k$, i.e., all proper faces of $F_k$.

Consider all the codimension one faces of the facet $F_w=(\emptyset, w, \emptyset)$. These come in four types:
\begin{itemize}
\item $(\{s\},sw,\emptyset)$ if $s \in \Des_L(w)$,
\item $(\emptyset,ws,\{s\})$ if $s \in \Des_R(w)$,
\item $(\{s\},w,\emptyset)$ if $s \in \Asc_L(w)$,
\item $(\emptyset,w,\{s\})$ if $s \in \Asc_R(w)$.
\end{itemize}
In the first two cases, the elements $sw$ and $ws$ are below $w$ in the two-sided weak order. If we order the facets of $\Xi$ according to a linear extension of the two-sided weak order:
\[
F_{w_1}, F_{w_2}, \ldots, F_{w_k}, F_w, \ldots,
\]
then the intersection of the boundary of $F_w$ with the union of the prior facets is given by those faces below $F_w$ in $\Xi$ that are not represented by $w$, i.e.,
\[
 \partial F_w \cap \left(\bigcup_{i=1}^k \partial F_{w_i} \right) = \bigcup_{\substack{ s \in \Des_L(w) \\ t \in \Des_R(w)}} [(S,e,S), (\{s\}, sw, \emptyset)] \cup [(S,e,S), (\emptyset, wt, \{t\})].
\]
Because all maximal faces have codimension one, we have proved the following proposition.

\begin{prp}[Shelling order]
Any linear extension of the two-sided weak order on $W$ is a shelling order for $\Xi$. In particular, any linear extension of the Bruhat order is a shelling order.
\end{prp}

This proves part (3) of Theorem \ref{thm:main}. A shelling of $\Xi(A_2)$ is indicated in Figure \ref{fig:A2}. The highlighted edges represent the intervals $[R_w,F_w]$, and with facets taken left to right, we have a linear extension of the two-sided weak order.

\subsection{Consequences of shelling}

%Let us say a bit more about the codimension one faces of $\Xi$. These are of the form $(\{s\},w,\emptyset)$ or $(\emptyset,w,\{t\})$. Without loss of generality, consider the first case. The double coset here has only two elements: $W_{\{s\}}wW_{\emptyset} = \{ w, sw\}$, so the face $(\{s\},w,\emptyset)$ is only contained in the facets $(\emptyset, w, \emptyset)$ and $(\emptyset, sw, \emptyset)$. This shows that $\Xi$ is  a \emph{pseudomanifold}.

A simplicial complex is a \emph{psuedomanifold} if every codimension one face is contained in exactly two maximal faces. A result of Bj\"orner tells us about shellable pseudomanifolds. 

\begin{thm}[Bj\"orner { \cite[Theorem 1.5]{Bjorner2} }]
Suppose $\Delta$ is a shellable pseudomanifold. If $\Delta$ is infinite, it is contractible. If $\Delta$ is finite it is a sphere.
\end{thm}
 
While $\Xi$ is not a simplicial complex, its barycentric subdivision is. Let $\Xi'$ denote this simplicial complex, whose faces are chains 
\[
 F'=\emptyset <_{\Xi} F_1 <_{\Xi} F_2 <_{\Xi} \cdots <_{\Xi} F_k, \qquad F_i \in \Xi.
\]
The dimension of such a face is $k-1$, and inclusion of faces in $\Xi'$ is given by inclusion of the sets of faces, i.e., $F' \leq_{\Xi'} G'=\emptyset  <_{\Xi} G_1  <_{\Xi}  \cdots  <_{\Xi} G_l$ if and only if
\[
 \{ F_1,\ldots,F_k\} \subseteq \{ G_1,\ldots,G_l\}.
\]

A poset is called \emph{thin} if every interval of length two has exactly four elements. Since $\Xi$ is a simplicial poset, every interval is boolean, and $\Xi$ is clearly thin. The nice thing about being thin is that the barycentric subdivision $\Xi'$ is a pseudomanifold. Indeed if $F'$ is a codimension one face of $\Xi'$ it has the form
\[
 F'= \emptyset <_{\Xi} F_1 <_{\Xi} F_2 <_{\Xi} \cdots <_{\Xi} F_{j-1} <_{\Xi} F_{j+1} <_{\Xi} \cdots <_{\Xi} F_d,
\]
where $\dim(F_i) = i-1$ and $d=2n$. Since $\Xi$ is thin the interval $[F_{j-1},F_{j+1}] = \{ F_{j-1}, H, H', F_{j+1}\}$ has exactly four elements, so there are exactly two choices for how to fill the gap in $F'$ to create a facet of $\Xi'$; either $F_{j-1} <_{\Xi} H <_{\Xi} F_{j+1}$ or $F_{j-1} <_{\Xi} H' <_{\Xi} F_{j+1}$. 

Here we are tacitly assuming $j=1,\ldots,d-1$, but we also need to consider the $j=d$ case, i.e., faces in $\Xi'$ of the form
\[
F'= \emptyset <_{\Xi} F_1 <_{\Xi} F_2 <_{\Xi} \cdots <_{\Xi} F_{d-1}.
\]
But if $F_{d-1}$ is a codimension one face of $\Xi$, we saw from Section \ref{sec:shelling} it has the form $(\{s\},w,\emptyset)$ or $(\emptyset,w,\{s\})$, whose corresponding double cosets are  $W_{\{s\}}wW_{\emptyset} = \{ w, sw\}$ or $W_{\emptyset}wW_{\{s\}} = \{ w, ws\}$. In either case, the coset has exactly two elements, so the face $(\{s\},w,\emptyset)$ is only contained in the facets $(\emptyset, w, \emptyset)$ and $(\emptyset, sw, \emptyset)$, while $(\emptyset,w,\{s\})$ is only contained in $(\emptyset,w,\emptyset)$ and $(\emptyset,ws,\emptyset)$. 

Thus we have shown that every codimension one face $F'$ of $\Xi'$ is contained in exactly two maximal faces, i.e., $\Xi'$ is a pseudomanifold.

Having established that $\Xi'$ is a pseudomanifold, we also claim that $\Xi'$ inherits shellability from $\Xi$. This is well-known for finite posets, see, e.g., \cite[Proposition 4.4(a)]{Bjorner}, and is easily generalized to arbitrary simplicial posets whose facets all have the same dimension.

To summarize, the barycentric subdivision of $\Xi$ is a shellable pseudomanifold. Since barycentric subdivision respects topology, we obtain the following corollary, establishing parts (4) and (5a) of Theorem \ref{thm:main}. 

\begin{cor}\label{cor:sphere}
The barycentric subdivision of $\Xi$ is a shellable pseudomanifold, and hence:
\begin{itemize}
\item $\Xi$ is contractible when $W$ is infinite,
\item $\Xi$ is a sphere when $W$ is finite.
\end{itemize}
\end{cor}

\begin{rem}
Let $\hat\Xi = \Xi \cup \{\hat 1\}$ be the poset obtained by adding a unique maximal element $\hat 1$ to the poset $\Xi$. Our argument for showing that $\Xi'$ is a pseudomanifold is essentially the argument that the poset $\hat\Xi$ is thin. The fact that $\Xi$ is a sphere in the finite case thus follows from \cite[Proposition 4.5]{Bjorner}.
\end{rem}

\section{Face enumeration for finite $W$}\label{sec:faces}

Throughout this section we assume $W$ is finite and fix an ordering on the generating set, $S=\{s_1,\ldots,s_n\}$. In this way we can identify subsets of $S$ with subsets of $[n]:=\{1,2,\ldots,n\}$. Let $x_1,\ldots,x_n$ and $y_1,\ldots,y_n$ be indeterminates. If $I \subseteq [n]$, let $x_I = \prod_{i \in I} x_i$, and similarly for $y_I$.  

For a face $F = (I, w, J)$ in $\Xi$, the \emph{face monomial} for $F$ is
\[
 m(F) = x_{[n]-I} y_{[n]-J}= \prod_{i \in [n]-I} x_i \prod_{j \in [n]-J} y_j.
\]
Notice this encodes the color of the face $F$; the $x$ variables encode the left sided vertices, the $y$ variables encode the right sided vertices.

Let $f(\mathbf{x},\mathbf{y})= f(x_1,\ldots,x_n,y_1,\ldots,y_n)$ denote the generating function for colors of faces, i.e.,
\[
 f(\mathbf{x},\mathbf{y}) = \sum_{F \in \Xi} m(F) = \sum_{I, J} f_{I,J} x_I y_J.
\]
Notice that the coefficient $f_{I,J}$ is the number of faces $(S-I,w,S-J)$, i.e., it counts the cardinality of the corresponding double quotient:
\begin{equation}\label{eq:double}
 f_{I,J} = |{}^{S-I}W^{S-J}|=|W_{S-I}\backslash W /W_{S-J}|.
\end{equation}
By Lemma \ref{lem:minrep} this is
\[
 f_{I,J} = |\{ w \in W: \Des_L(w) \subseteq I, \Des_R(w) \subseteq J\}|.
\]

Now define the quantities
\begin{align*}
 h_{I,J} &= \sum_{\substack{ K\subseteq I \\ L \subseteq J}} (-1)^{|I-K|+|J-L|} f_{K,L},\\
  &=|\{ w \in W: \Des_L(w) = I, \Des_R(w) = J\}|,
\end{align*}
and the corresponding generating function 
\begin{align*}
 h(x_1,\ldots,x_n,y_1,\ldots,y_n) &= \sum_{I,J} h_{I,J}x_Iy_J,\\
  &= \sum_{w \in W} x_{\Des_L(w)} y_{\Des_R(w)}.
\end{align*}

Recall that for a fixed element $w \in W$, the interval $[R_w,F_w]$ contains all the faces represented by $w$, and this interval is isomorphic to the boolean interval $\Asc_L(w)\times \Asc_R(w)$. This means the generating function for faces in this interval has the following form:
\begin{align*}
 \sum_{R_w \leq F \leq F_w} m(F) &= m(R_w)\cdot \prod_{i \in \Asc_L(w)} (1+x_i)\cdot  \prod_{ j \in \Asc_R(w) } (1+y_j),\\
  &= x_{\Des_L(w)}y_{\Des_R(w)} \cdot \prod_{i \in \Asc_L(w)} (1+x_i)\cdot \prod_{ j \in \Asc_R(w) } (1+y_j),\\
 &= \left( \prod_{i=1}^n (1+x_i)(1+y_i) \right)\prod_{ j \in \Des_L(w)} \frac{x_j}{1+x_j} \prod_{k \in \Des_R(w)} \frac{y_k}{1+y_k},
\end{align*}
where the final equality comes from the fact that ascent sets and descent sets are complementary.

Now using the partitioning of faces of $\Xi$ given in \eqref{eq:partition}, we get\begin{align}
 f(\mathbf{x},\mathbf{y}) &= \sum_{F \in \Xi} m(F),\nonumber \\
  &= \sum_{w \in W} \sum_{R_w \leq F \leq F_w} m(F),\nonumber \\
  &=  \prod_{i=1}^n (1+x_i)(1+y_i) \sum_{w \in W} \prod_{ j \in \Des_L(w)} \frac{x_j}{1+x_j} \prod_{k \in \Des_R(w)} \frac{y_k}{1+y_k},\nonumber \\
 &= \prod_{i =1}^n (1+x_i)(1+y_i) h\left( \frac{x_1}{1+x_1}, \ldots,\frac{x_n}{1+x_n}, \frac{y_1}{1+y_1}, \ldots, \frac{y_n}{1+y_n} \right). \label{eq:ftoh}
\end{align}
That is, we obtain the $f$-polynomial as a multiple of a certain specialization of the $h$-polynomial. Putting identity \eqref{eq:ftoh} the other way around, we can write
\begin{equation}\label{eq:htof}
 h(\mathbf{x},\mathbf{y}) = \prod_{i=1}^n (1-x_i)(1-y_i) f\left( \frac{x_1}{1-x_1}, \ldots,\frac{x_n}{1-x_n}, \frac{y_1}{1-y_1}, \ldots, \frac{y_n}{1-y_n} \right).
\end{equation}

Setting $x_j=x$ and $y_k=y$, we have 
\[
f(x,y) = \sum_{F \in \Xi} x^{l(F)}y^{r(F)},
\]
where if $F=(J,w,K)$, $l(F) = |S-J|$ and $r(F) = |S-K|$, which counts faces according to the number of ``left" and ``right" vertices. The $h$-polynomial specializes to
\[
 h(x,y) = \sum_{w \in W} x^{\des_L(w)} y^{\des_R(w)}.
\]
In other words, the polynomial $h(x,y)$ is a ``two-sided" Eulerian polynomial. This establishes the claim in part (5b) of Theorem \ref{thm:main}.

The usual $f$- and $h$-polynomials of $\Xi$ can be obtained by the further specialization of $x=y$:
\[
 f(x) = \sum_{F \in \Xi} x^{|F|}, \qquad h(x) = \sum_{w \in W} x^{\des_L(w) + \des_R(w)}.
\]

\begin{example} We can see in Figure \ref{fig:A2} that
\begin{align*}
 f(A_2;\mathbf{x},\mathbf{y}) &= 1 + (x_1 + x_2 + y_1 + y_2) \\
 &+ (x_1x_2 + 2x_1y_1 + 2x_1y_2 + 2x_2y_1 + 2x_2y_2 + y_1y_2) \\
  &+ (3x_1x_2y_1 + 3x_1x_2y_2 + 3x_1y_1y_2 + 3x_2y_1y_2) + 6x_1x_2y_1y_2,
\end{align*}
Which after a bit of rearranging equals
\begin{align*}
  &(1+x_1)(1+x_2)(1+y_1)(1+y_2) \\
  &+ x_1y_1(1+x_2)(1+y_2) + x_1y_2(1+x_2)(1+y_1) \\
  &+ x_2y_1(1+x_1)(1+y_2) + x_2y_2(1+x_1)(1+y_1)\\
  &+ x_1x_2y_1y_2.
\end{align*}

The elements of $A_2$ have the following descent sets,
\[
 \begin{array}{c | c | c}
  w & \Des_L(w) & \Des_R(w) \\
  \hline
  e & \emptyset & \emptyset\\
  s_1 & \{1\} & \{1\} \\
  s_2 & \{2\} & \{2\} \\
  s_1s_2 & \{1\} & \{2\}\\
  s_2s_1 & \{2\} & \{1\} \\
  s_1s_2s_1 = s_2 s_1s_2 & \{1,2\} & \{1,2\}  
 \end{array}
\]
so we can see that 
\begin{align*}
 f(A_2; \mathbf{x},\mathbf{y}) &= \sum_{w \in A_2} x_{\Des_L(w)}y_{\Des_R(w)}\prod_{i \in \Asc_L(w)} (1+x_i) \cdot \prod_{j \in \Asc_R(w)} (1+y_j),\\
  &= h\left(A_2; \frac{x_1}{1+x_1}, \frac{x_2}{1+x_2},\frac{y_1}{1+y_1},\frac{y_2}{1+y_2}\right).
\end{align*}

The coarser polynomials are then
\[ f(A_2; x, y) = 1 + 2(x+y) + x^2 + 8xy + y^2 + 6(x^2y + xy^2) + 6x^2y^2,\]
and
\[ h(A_2; x,y) = 1 + 4xy + x^2y^2.\]
\end{example}

\section{Two-sided Eulerian polynomials}\label{sec:poly}

With finite $W$, we can define the \emph{two-sided $W$-Eulerian polynomial}, denoted $W(x,y)$, as the joint distribution of left and right descents:
\[
W(x,y) =\sum_{w \in W} x^{\des_L(w)}y^{\des_R(w)} = \sum_{0\leq i,j\leq n} \el{W}{i,j} x^i y^j,
\] 
where $\el{W}{i,j}$ denotes the number of elements in $W$ with $i$ left descents and $j$ right descents. We call $\el{W}{i,j}$ a \emph{two-sided $W$-Eulerian number}. In Tables \ref{tab:WpolyABD} and \ref{tab:WpolyEF} we have the arrays of coefficients
\[
 \left[ \el{W}{i,j} \right]_{0\leq i,j \leq n},
\]
for some finite Coxeter groups of small rank. 

\begin{table}
\[
\begin{array}{| c  c |}
\hline  & \\
W & \left[ \el{W}{i,j} \right]_{0\leq i,j \leq n} \\
& \\
\hline
\hline & \\
A_n\, (n\geq 1): &  {\small \left[ \begin{array}{r r} 1 & 0 \\ 0 & 1\end{array} \right], \left[ \begin{array}{r rr} 1 & 0 &0 \\ 0 & 4 & 0 \\ 0 & 0 & 1 \end{array} \right], \left[ \begin{array}{r rrr} 1 & 0 &0 &0 \\ 0 & 10 & 1 & 0 \\ 0 & 1 & 10 & 0 \\ 0 & 0 & 0 & 1 \end{array} \right], \left[ \begin{array}{r rrrr} 1 & 0 &0 &0& 0\\ 0 & 20 & 6 & 0 & 0\\ 0 & 6 & 54 & 6 & 0 \\ 0 & 0 & 6 & 20& 0 \\ 0 & 0 & 0 & 0 &1 \end{array} \right]} \\
& \\
\hline
& \\
B_n\, (n\geq 2): & {\small \left[ \begin{array}{r rr} 1 & 0 &0 \\ 0 & 6 & 0 \\ 0 & 0 & 1 \end{array} \right], \left[ \begin{array}{r rrr} 1 & 0 &0 &0 \\ 0 & 19 & 4 & 0 \\ 0 & 4 & 19 & 0 \\ 0 & 0 & 0 & 1 \end{array} \right], \left[ \begin{array}{r rrrr} 1 & 0 &0 &0& 0\\ 0 & 45 & 30 & 1 & 0\\ 0 & 30 & 170 & 30 & 0 \\ 0 & 1 & 30 & 45& 0 \\ 0 & 0 & 0 & 0 &1 \end{array} \right] }\\
& \\
\hline
& \\
D_n\, (n\geq 4): & {\small \left[ \begin{array}{r rrrr} 1 & 0 &0 &0& 0\\ 0 & 30 & 12 & 2 & 0\\ 0 & 12 & 78 & 12 & 0 \\ 0 & 2 & 12 & 30 & 0 \\ 0 & 0 & 0 & 0 &1 \end{array} \right], \left[ \begin{array}{r rrrrr} 1 & 0 &0 &0& 0 & 0\\ 0 & 69 & 69 & 18 & 1 & 0\\ 0 & 69 & 486 & 229 & 18 & 0 \\ 0 & 18 & 229 & 486 & 69 & 0 \\ 0 & 1 & 18 & 69 & 69 & 0 \\ 0 & 0 & 0 & 0 & 0 & 1 \end{array} \right]},\\
& \\
& {\small\left[ \begin{array}{rrrrrrr}  1 & 0 & 0 & 0 & 0 & 0 & 0 \\ 0 & 135 & 262 & 117 & 16 & 0 & 0 \\ 0 & 262 & 2433 & 2330 & 510 & 16 & 0 \\ 0 & 117 & 2330 & 5982 & 2330 & 117 & 0 \\ 0 & 16 & 510 & 2330 & 2433 & 262 & 0 \\ 0 & 0 & 16 & 117 & 262 & 135 & 0 \\ 0 & 0 & 0 & 0 & 0 & 0 & 1 \end{array} \right]}\\
& \\
\hline
\end{array}
\]
\caption{Two-sided Eulerian numbers of low rank for finite Coxeter groups of classical type (types $A_n, B_n, D_n$).}\label{tab:WpolyABD}
\end{table}

\begin{table}
\[
\begin{array}{| c  c |}
\hline  & \\
W & \left[ \gamma^{W}_{a,b} \right]_{0\leq 2a+b \leq n} \\
& \\
\hline
\hline & \\
A_n\, (n\geq 1): &  {\small \left[ \begin{array}{r r} 1 \end{array} \right], \left[ \begin{array}{r rr} 1 & 0 \\ 0 & 2 \end{array} \right], \left[ \begin{array}{r rrr} 1 & 0 \\ 0 & 7 \\ 0 & 1 \end{array} \right], \left[ \begin{array}{r rrrr} 1 & 0 &0 \\ 0 & 16 & 0 \\ 0 & 6 & 16  \end{array} \right]} \\
& \\
\hline
& \\
B_n\, (n\geq 2): & {\small \left[ \begin{array}{r rr} 1 & 0 \\ 0 & 4 \end{array} \right], \left[ \begin{array}{r rrr} 1 & 0 \\ 0 & 16  \\ 0 & 4  \end{array} \right], \left[ \begin{array}{r rrrr} 1 & 0 & 0  \\ 0 & 41 & 0  \\ 0 & 30 & 80  \\ 0 & 1 & 0  \end{array} \right] }\\
& \\
\hline
& \\
D_n\, (n\geq 4): & {\small \left[ \begin{array}{r rrrr} 1 & 0 & 0  \\ 0 & 26 & 0  \\ 0 & 12 & 16  \\ 0 & 2 & 0  \end{array} \right], \left[ \begin{array}{r rrrrr} 1 & 0 & 0  \\ 0 & 64 & 0  \\ 0 & 69 & 248 \\ 0 & 18 & 88  \\ 0 & 1 & 0  \end{array} \right]},{\small\left[ \begin{array}{rrrrrrr}  1 & 0 & 0 & 0  \\ 0 & 129 & 0 & 0  \\ 0 & 262 & 1668 & 0  \\ 0 & 117 & 1496 & 832 \\ 0 & 16 & 276 & 0  \end{array} \right]}\\
& \\
\hline
\end{array}
\]
\caption{The integers $\gamma^{W}_{a,b}$ of low rank for finite Coxeter groups of classical type (types $A_n, B_n, D_n$).}\label{tab:WpolyABD2}
\end{table}

\begin{table}
\[
\begin{array}{| c  c |}
\hline & \\
W & \left[ \el{W}{i,j} \right]_{0\leq i,j \leq n} \\
& \\
\hline
\hline & \\
E_6: &  {\small \left[ \begin{array}{rrr rrrr} 1 & 0 &0 &0& 0& 0 &0\\ 0 & 232 & 584 & 389 & 64& 3 & 0\\ 0 & 584 & 4785 & 5440 & 1310 & 64 & 0 \\ 0 & 389 & 5440 & 13270 & 5440 & 389 &0 \\ 0 & 64 & 1310 & 5440 & 4785 & 584 & 0 \\ 0 & 3 & 64 & 389 & 584 & 232 & 0 \\ 0 & 0 & 0 & 0 & 0 & 0 & 1 \end{array} \right]}\\
& \\
\hline
& \\
E_7: &
{\small \left[
\begin{array}{rrrrrrrr} 1& 0& 0& 0& 0& 0& 0& 0 \\ 0& 945& 5414& 7693& 3208& 367& 8& 0 \\ 0& 5414& 64905& 143036& 83491& 12756& 367& 0 \\ 0& 7693& 143036& 484551& 401936& 83491& 3208& 0 \\ 0& 3208& 83491& 401936& 484551& 143036& 7693& 0 \\ 0& 367& 12756& 83491& 143036& 64905& 5414& 0 \\ 0& 8& 367& 3208& 7693& 5414& 945& 0 \\ 0& 0& 0& 0& 0& 0& 0& 1
\end{array}
\right]}\\
& \\
\hline
& \\
E_8: & {\tiny \left[ \begin{array}{rrrrr rrrr} 1 & 0 & 0 & 0 & 0 & 0 & 0 & 0 & 0 \\ 0 & 8460 & 113241 & 338944 & 318372 & 94540 & 8103 & 92 & 0 \\ 0 & 113241 & 2348364 & 9509809 & 11520216 & 4360423 & 476192 & 8103 & 0 \\ 0 & 338944 & 9509809 & 48819660 & 72638788 & 33260660 & 4360423 & 94540 & 0 \\ 0 & 318372 & 11520216 & 72638788 & 131292998 & 72638788 & 11520216 & 318372 & 0 \\ 0 & 94540 & 4360423 & 33260660 & 72638788 & 48819660 & 9509809 & 338944 & 0 \\ 0 & 8103 & 476192 & 4360423 & 11520216 & 9509809 & 2348364 & 113241 & 0 \\ 0 & 92 & 8103 & 94540 & 318372 & 338944 & 113241 & 8460 & 0 \\ 0 & 0 & 0 & 0 & 0 & 0 & 0 & 0 & 1  \end{array} \right]} \\
& \\
\hline
& \\
F_4: & \left[ \begin{array}{rrr rrrr} 1 & 0 & 0 & 0 & 0 \\ 0 & 108 & 112 & 16 & 0 \\ 0 & 112 & 454 & 112 & 0 \\ 0 & 16 & 112 & 108 & 0 \\ 0 & 0 & 0 & 0 & 1 \end{array} \right]\\
& \\
\hline
\end{array}
\]
\caption{The two-sided Eulerian numbers for finite Coxeter groups of exceptional type.}\label{tab:WpolyEF}
\end{table}

\begin{table}
\[
\begin{array}{| c  c |}
\hline & \\
W & \left[ \gamma^{W}_{a,b} \right]_{0\leq 2a+b \leq n} \\
& \\
\hline
\hline & \\
E_6: &  {\small \left[ \begin{array}{rrr rrrr} 1 & 0 & 0 & 0 \\ 0 & 226 & 0 & 0  \\ 0 & 584 & 3088 & 0 \\ 0 & 389 & 3496 & 3104 \\ 0 & 64 & 520 & 0  \\ 0 & 3 & 0 & 0 \end{array} \right]}\\
& \\
\hline
& \\
E_7: &
{\small \left[
\begin{array}{rrrrrrrr} 1 & 0 & 0 & 0 \\ 0 & 938 & 0 & 0  \\ 0 & 5414 & 44808 & 0  \\ 0 & 7693 & 111756 & 174464  \\ 0 & 3208 & 58944 & 107712 \\ 0 & 367 & 6300 & 0  \\ 0 & 8 & 0 & 0 \end{array}
\right]}\\
& \\
\hline
& \\
E_8: & {\small \left[ \begin{array}{rrrrr rrrr} 1 & 0 & 0 & 0 & 0  \\ 0 & 8452 & 0 & 0 & 0  \\ 0 & 113241 & 1619736 & 0 & 0  \\ 0 & 338944 & 7988488 & 19362528 & 0  \\ 0 & 318372 & 9786280 & 34500160 & 17111296  \\ 0 & 94540 & 3364792 & 9750496 & 0 \\ 0 & 8103 & 286560 & 0 & 0  \\ 0 & 92 & 0 & 0 & 0 \end{array} \right]} \\
& \\
\hline
& \\
F_4: & \left[ \begin{array}{rrr rrrr} 1 & 0 & 0 & 0 \\ 0 & 104 & 0 & 0 \\ 0 & 112 & 208 & 0 \\ 0 & 16 & 0 & 0 \end{array} \right]\\
& \\
\hline
\end{array}
\]
\caption{The integers $\gamma^{W}_{i,j}$ of low rank for finite Coxeter groups of exceptional type.}\label{tab:WpolyEF2}
\end{table}

In type $A_n$, these numbers were first studied by Carlitz et al. \cite{Carlitz}, but have been recently revisited by the author \cite{Petersen} and Visontai \cite{Visontai} (who also discussed type $B_n$ Coxeter groups). The recent interest in these polynomials stems from a conjecture of Gessel that we will now describe and generalize from the symmetric group to all finite Coxeter groups. 

To state Gessel's conjecture, one must first make note of certain symmetries in the two-sided Eulerian numbers. Notice that the map $w \mapsto w^{-1}$ swaps left and right descents, $\Des_L(w) = \Des_R(w^{-1})$, so we get symmetry in $i$ and $j$:
\begin{equation}\label{eq:ijsym}
 \el{W}{i,j} = \el{W}{j,i}.
\end{equation}
Also recall that left multiplication by the long element $w_0$ complements the right descent set:
\[
 \Des_R(w_0 w) = S- \Des_R(w),
\]
while conjugation by $w_0$ conjugates the elements of the right descent set:
\[
 \Des_R(w_0 w w_0) = \{ w_0sw_0 : s \in \Des_R(w)\} = w_0\Des_R(w) w_0.
\] 
These facts follow, e.g., from \cite[Section 2.3]{BjB}.

Taken together, we see that left multiplication by $w_0$ complements the conjugate of the left descent set:
\begin{align*}
\Des_L(w_0w) &=\Des_R(w^{-1}w_0),\\
 &=\Des_R(w_0(w_0w^{-1}w_0)),\\
 &=S-\Des_R(w_0w^{-1}w_0),\\
 &=S-w_0\Des_R(w^{-1})w_0,\\
 &=S-w_0\Des_L(w) w_0.
\end{align*}

%Thus if $|S|=n$, $\des_R(w_0w) = n-\des_R(w)$, and since $w_0$ is its own inverse,
%\[
%\des_L(w_0w) = \des_R(w^{-1}w_0) = n-\des_R(w_0 w^{-1} w_0) = n-\des_L(w_0 w w_0).
%\]
%Taking $u = w_0w$, this can be expressed as
%\[
% \des_L(u) = n-\des_L(uw_0).
%\]
%
%Thus if $v = uw_0$, we have
%\[
%  \des_L(v) = n-\des_L(u) \quad \mbox{and} \quad \des_R(v) = \des_L(v^{-1}) = \des_L(w_0 u^{-1}) =
%\] 

Hence we have $\des_L(w_0w) = n-\des_L(w)$ and $\des_R(w_0w) = n- \des_R(w)$, implying the following symmetry:
\begin{equation}\label{eq:n-isym}
 \el{W}{i,j} = \el{W}{n-i,n-j}.
\end{equation}
Phrasing symmetries \eqref{eq:ijsym} and \eqref{eq:n-isym} in terms of generating functions, we have the following observation about the two-sided $W$-Eulerian polynomials.

\begin{obs}
For any finite Coxeter group $W$ of rank $n$,
\begin{enumerate}
\item $W(x,y) = W(y,x)$, and
\item $W(x,y)=x^ny^nW(1/x,1/y)$.
\end{enumerate}
\end{obs}

Integer polynomials that possess symmetries (1) and (2) have an expansion in the following basis:
\[
 \Gamma_n=\{ (x y)^a (x+y)^b (1+xy)^{n-2a-b} \}_{0\leq 2a + b \leq n}.
\]
The generalized Gessel conjecture is that the two-sided Eulerian polynomials expand positively in this basis.

\begin{conj}[Generalized Gessel's conjecture]\label{conj:gamma}
For any finite Coxeter group $W$ of rank $n$, there exist nonnegative integers $\gamma_{a,b}^W$ such that
\[
 W(x,y) = \sum_{0\leq 2a+b \leq n} \gamma_{a,b}^W (xy)^a(x+y)^b(1+xy)^{n-2a-b}.
\]
\end{conj}

The integers $\gamma_{a,b}^W$ for $W$ of small rank are shown in Tables \ref{tab:WpolyABD2} and \ref{tab:WpolyEF2}.

\begin{rem}
In practice, traversing the group $W$ to compute the polynomial $W(x,y)$ is not very efficient, as the order of the group is roughly factorial in the rank. 

From Equation \eqref{eq:double} we know $f_{S-I,S-J}$ is the cardinality of the double quotient $|W_I\backslash W/W_J|$ and from \cite[Exercise 7.77a]{ec2} we can compute this cardinality with an inner product of trivial characters on the parabolic subgroups induced up to $W$. That is,
\[
|W_I\backslash W/W_J| = \left\langle \mathrm{ind}_{W_{I}}^{W}1_{W_{I}}, \mathrm{ind}_{W_{J}}^{W}1_{W_{J}} \right\rangle,
\]
where $1_{W_J}$ denotes the trivial character on $W_J$. Stembridge has a nice implementation of this character computation in Maple \cite{StM}. 

Having computed the numbers $f_{I,J}$ for all pairs of subsets $I, J \subseteq S$, we obtain the polynomial $f(\mathbf{x},\mathbf{y})$ and we can use Equation \eqref{eq:htof} to compute the polynomial $h(\mathbf{x},\mathbf{y})$, which then specializes to $W(x,y)$. To put it succinctly, we have
\[
 W(x,y) = \sum_{I,J \subseteq S} f_{I,J}x^{|I|}y^{|J|}(1-x)^{n-|I|}(1-y)^{n-|J|}.
\]

Roughly speaking, this method reduces the problem of computing $W(x,y)$ from that of traversing the $|W|$ elements of $W$ to one of traversing $4^n$ pairs of subsets.  The two-sided Eulerian numbers for $E_8$ were computed in about half an hour on a standard desktop machine in this manner.
\end{rem}

\begin{rem}
Very recently, the author was informed that Gessel's original conjecture (for $W=A_n = S_{n+1}$) was proved by Lin \cite{Lin}. The method of proof seems to be a careful induction argument using a recurrence for the $\gamma_{a,b}^{A_n}$ given by Visontai \cite{Visontai}. The other cases have been verified for small rank $(n\leq 10)$. Type $B_n$ is governed by similar combinatorics, so perhaps a similar induction proof can be found. In all cases, it would be nice to know what the $\gamma_{a,b}^W$ count.
\end{rem}

\section{Contingency tables}\label{sec:contingency}
 
Throughout this section we consider the special case where $W=S_n$ is the symmetric group. The generating set is $S=\{s_1,s_2,\ldots,s_{n-1}\}$, where $s_i$ is the $i$th adjacent transposition.

As shown in Diaconis and Gangolli \cite{DG}, for fixed $I$ and $J$ the double cosets $W_I w W_J$ are in bijection with arrays of nonnegative integers. (They attribute the idea to N. Bergeron.) To see how this connection is made, we draw double cosets as diagrams of ``balls in boxes." First, we draw permutations as two-dimensional arrays, with a ball in column $i$ (left to right), row $j$ (bottom to top), if $w(i)=j$, then we insert some vertical and horizontal bars in gaps between balls. The group $S_n$ acts on the left by permuting rows; it acts on the right by permuting columns.

For example, $w = 7142536$ is drawn in Figure \ref{fig:double}. To indicate a parabolic double coset $W_IwW_J$, we draw solid horizontal bars in gaps that correspond to $S-I$ and solid vertical bars in gaps that correspond to $S-J$. In Figure \ref{fig:double}, $I=\{s_1,s_2,s_3, s_5\}$ and $J=\{s_2,s_3,s_6\}$. We can get all elements of $W_IwW_J$ by swapping columns and rows that are not separated by a solid bar. Notice that the balls cannot leave the boxes formed by the bars. 

The minimal representative for the double coset corresponds to the permutation obtained by sorting the balls in increasing order from left to right and from bottom to top. The minimal representative for the coset illustrated in Figure \ref{fig:double} would then be $u = 7123546$. Notice that both the right descents and left descents of $u$ occur in barred positions. 

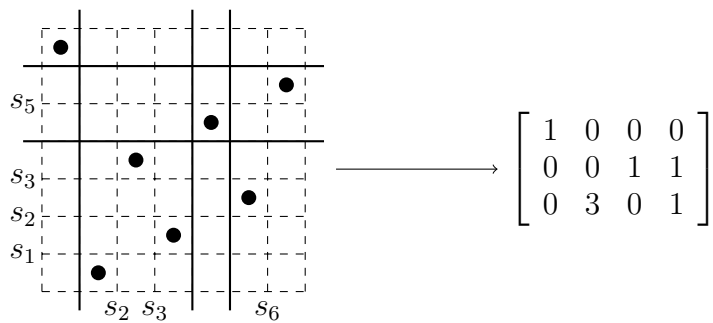
\begin{figure}
\[
\begin{tikzpicture}
\draw (0,0) node (a) {
 \begin{tikzpicture}[scale=.25]
  \draw[dashed] (0,0) grid[step=2] (14,14); 
  \draw (1,13) node[circle,fill=black,inner sep =2] {};
  \draw (3,1) node[circle,fill=black,inner sep =2] {};
  \draw (5,7) node[circle,fill=black,inner sep =2] {};
  \draw (7,3) node[circle,fill=black,inner sep =2] {};
  \draw (9,9) node[circle,fill=black,inner sep =2] {};
  \draw (11,5) node[circle,fill=black,inner sep =2] {};
  \draw (13,11) node[circle,fill=black,inner sep =2] {};
  \draw[thick] (2,-1)--(2,15);
  \draw[thick] (8,-1)--(8,15);
  \draw[thick] (10,-1)--(10,15);
  \draw[thick] (-1,8)--(15,8);
  \draw[thick] (-1,12)--(15,12);
  \draw (4,-1) node {$s_2$};
  \draw (6,-1) node {$s_3$};
  \draw (12,-1) node {$s_6$};
  \draw (-1,2) node {$s_1$};
  \draw (-1,4) node {$s_2$};
  \draw (-1,6) node {$s_3$};
  \draw (-1,10) node {$s_5$};
%  \draw (7,-3) node {(right action)};
%  \draw (-3,7) node[rotate=-90] {(left action)};
 \end{tikzpicture}
 };
 \draw (6,0) node (b) {
 $\left[ \begin{array}{ r r r r} 1 & 0 & 0 &0\\ 0 & 0 & 1 & 1 \\ 0 & 3 & 0 & 1 \end{array}\right]$
 };
 \draw[->] (a)--(b);
 \end{tikzpicture}
\]
\caption{A double coset in $A_6$ mapping to a contingency table in $\Xi(7)$.}\label{fig:double}
\end{figure}

Given the diagram for a double coset as in Figure \ref{fig:double}, we can map the diagram to an array of nonnegative integers by merely counting the number of balls in each box. Let $\Xi(n)$ denote the set of all such arrays, which are known as \emph{two-way contingency tables}. More precisely, define $\Xi(n)$ to be the set of all nonnegative integer arrays whose entries sum to $n$ and whose row sums and column sums are positive.
%\[
% \Xi(n) = \left\{ A=[a_{i,j}] : \begin{array}{l} a_{i,j} \geq 0, \\
%   \sum_{i,j} a_{i,j} = n,\\
%   r_i = \sum_j a_{i,j} > 0,\\
%    c_j = \sum_i a_{i,j} > 0 \end{array}
%    \right\}
%\]

To move up in the partial order, we refine our balls and boxes picture by inserting more bars. On the contingency table side, this means our arrays get more rows and columns. Each cover relation corresponds to adding or deleting a single bar, so rank is given by the total number of bars. A balls-in-boxes picture with $k$ horizontal bars and $l$ vertical bars will correspond to a $(k+1)\times(l+1)$ contingency table.

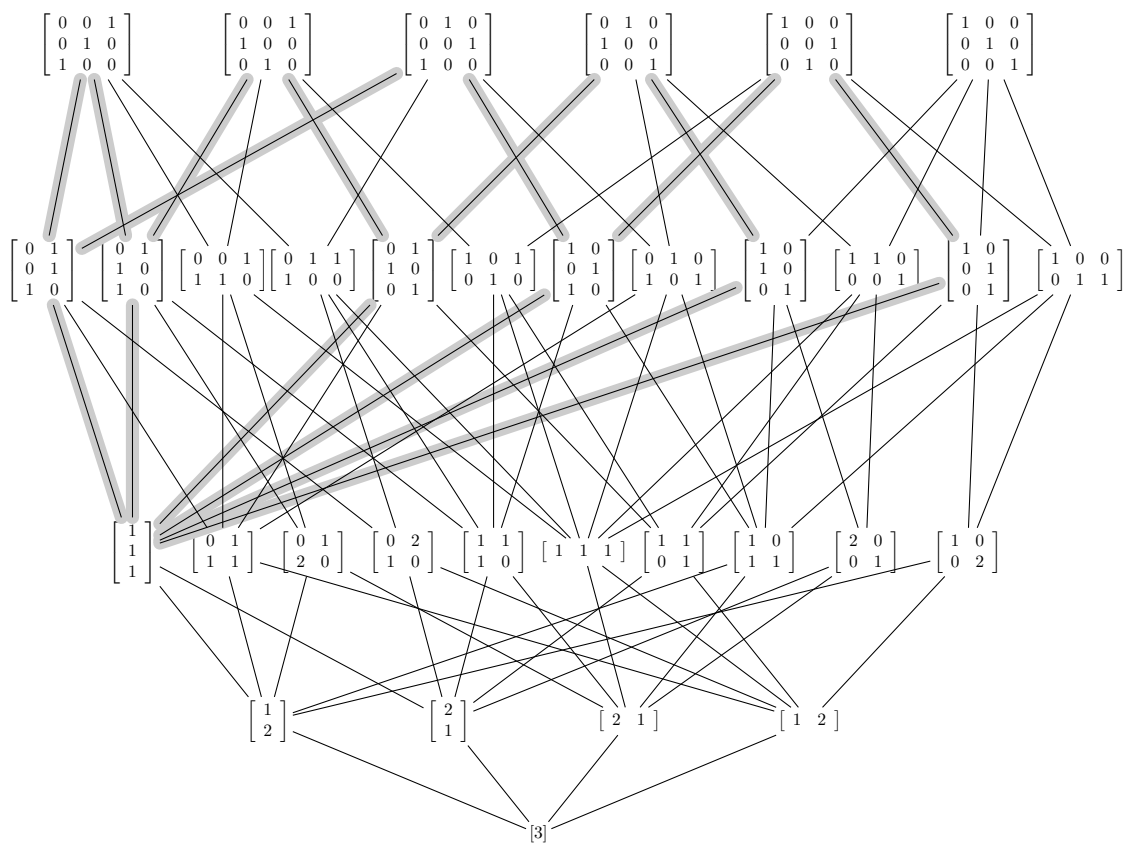
\begin{figure}
{\small
\[
 \begin{tikzpicture}[xscale=.6,yscale=3]
 \tikzstyle{state}=[fill=white,inner sep = 2,scale=.6];
 \tikzstyle{linet}=[line width=5, cap=round, color=white!80!black];
  \draw (0,.5) node[state] (12e12) {$[3]$};
  \draw (-6,1) node[state] (1e12) {$\left[ \begin{array}{c} 1 \\ 2 \end{array} \right]$};
  \draw (-2,1) node[state] (2e12) {$\left[ \begin{array}{c} 2 \\ 1 \end{array} \right]$};
  \draw (6,1) node[state] (12e2) {$\left[ \begin{array}{cc} 1 & 2 \end{array} \right]$};
  \draw (2,1) node[state] (12e1) {$\left[ \begin{array}{cc} 2 & 1 \end{array} \right]$};
  \draw (-9,1.75) node[state] (e12) {$\left[ \begin{array}{c} 1 \\ 1 \\ 1 \end{array} \right]$};
  \draw (-7,1.75) node[state] (1e2) {$\left[ \begin{array}{cc} 0 & 1 \\ 1 & 1\end{array} \right]$};
  \draw (-5,1.75) node[state] (1e1) {$\left[ \begin{array}{cc} 0 & 1 \\ 2 & 0 \end{array} \right]$};
  \draw (-3,1.75) node[state] (2e2) {$\left[ \begin{array}{cc} 0 & 2 \\ 1 & 0\end{array} \right]$};
  \draw (-1,1.75) node[state] (2e1) {$\left[ \begin{array}{cc} 1 & 1 \\ 1 & 0 \end{array} \right]$};
  \draw (1,1.75) node[state] (12e) {$\left[ \begin{array}{ccc} 1 & 1 & 1 \end{array} \right]$};
  \draw (3,1.75) node[state] (2s12) {$\left[ \begin{array}{cc} 1 & 1 \\ 0 & 1 \end{array} \right]$};
  \draw (5,1.75) node[state] (1s21) {$\left[ \begin{array}{cc} 1 & 0 \\ 1 & 1 \end{array} \right]$};
  \draw (7.25,1.75) node[state] (2s1s21) {$\left[ \begin{array}{cc} 2 & 0 \\ 0 & 1 \end{array} \right]$};
  \draw (9.5,1.75) node[state] (1s2s12) {$\left[ \begin{array}{cc} 1 & 0 \\ 0 & 2 \end{array} \right]$};
  \draw (-11,3) node[state] (e2) {$\left[ \begin{array}{cc} 0 & 1 \\ 0 & 1 \\ 1 & 0 \end{array} \right]$};
  \draw (-9,3) node[state] (e1) {$\left[ \begin{array}{cc} 0 & 1 \\ 1 & 0 \\ 1 & 0 \end{array} \right]$};
  \draw (-7,3) node[state] (1e) {$\left[ \begin{array}{ccc} 0 & 0 & 1 \\ 1 & 1 & 0 \end{array} \right]$};
  \draw (-5,3) node[state] (2e) {$\left[ \begin{array}{ccc} 0 & 1 & 1 \\ 1 & 0 & 0 \end{array} \right]$};
  \draw (-3,3) node[state] (s12) {$\left[ \begin{array}{cc} 0 & 1 \\ 1 & 0 \\ 0 & 1 \end{array} \right]$};
  \draw (-1,3) node[state] (2s1) {$\left[ \begin{array}{ccc} 1 & 0 & 1 \\ 0 & 1 & 0 \end{array} \right]$};
  \draw (1,3) node[state] (s21) {$\left[ \begin{array}{cc} 1 & 0 \\ 0 & 1 \\ 1 & 0 \end{array} \right]$};
  \draw (3,3) node[state] (1s2) {$\left[ \begin{array}{ccc} 0 & 1 & 0 \\ 1 & 0 & 1 \end{array} \right]$};
  \draw (5.25,3) node[state] (s1s21) {$\left[ \begin{array}{cc} 1 & 0 \\ 1 & 0 \\ 0 & 1 \end{array} \right]$};
  \draw (7.5,3) node[state] (2s1s2) {$\left[ \begin{array}{ccc} 1 & 1 & 0 \\ 0 & 0 & 1 \end{array} \right]$};
  \draw (9.75,3) node[state] (s2s12) {$\left[ \begin{array}{cc} 1 & 0 \\ 0 & 1 \\ 0 & 1 \end{array} \right]$};
  \draw (12,3) node[state] (1s2s1) {$\left[ \begin{array}{ccc} 1 & 0 & 0 \\ 0 & 1 & 1 \end{array} \right]$};
  \draw (-10,4) node[state] (e) {$\left[ \begin{array}{ccc} 0 & 0 & 1 \\ 0 & 1 & 0 \\ 1 & 0 & 0  \end{array} \right]$};
  \draw (-6,4) node[state] (s1) {$\left[ \begin{array}{ccc} 0 & 0 & 1 \\ 1 & 0 & 0 \\ 0 & 1 & 0 \end{array} \right]$};
  \draw (-2,4) node[state] (s2) {$\left[ \begin{array}{ccc} 0 & 1 & 0 \\ 0 & 0 & 1 \\ 1 & 0 & 0 \end{array} \right]$};
  \draw (2,4) node[state] (s1s2) {$\left[ \begin{array}{ccc} 0 & 1 & 0 \\ 1 & 0 & 0 \\ 0 & 0 & 1 \end{array} \right]$};
  \draw (6,4) node[state] (s2s1) {$\left[ \begin{array}{ccc} 1 & 0 & 0 \\ 0 & 0 & 1 \\ 0 & 1 & 0 \end{array} \right]$};
  \draw (10,4) node[state] (s1s2s1) {$\left[ \begin{array}{ccc} 1 & 0 & 0 \\ 0 & 1 & 0 \\ 0 & 0 & 1 \end{array} \right]$};
  \draw[linet] (e)--(e2);
  \draw[linet] (e)--(e1);
%  \draw[linet] (e)--(1e);
%  \draw[linet] (e)--(2e);
  \draw[linet] (e2)--(e12);
%  \draw[linet] (e2)--(1e2);
%  \draw[linet] (e2)--(2e2);
  \draw[linet] (e1)--(e12);
  \draw[linet] (e12)--(s12);
  \draw[linet] (s1)--(e1);
  \draw[linet] (s2)--(e2);
  \draw[linet] (s12)--(s1s2);
  \draw[linet] (e12)--(s21);
  \draw[linet] (s21)--(s2s1);
  \draw[linet] (e12)--(s2s12);
  \draw[linet] (e12)--(s1s21);
%  \draw[linet] (e1)--(1e1);
%  \draw[linet] (e1)--(2e1);
%  \draw[linet] (1e)--(1e2);
%  \draw[linet] (1e)--(1e1);
%  \draw[linet] (1e)--(12e);
%  \draw[linet] (2e)--(2e2);
%  \draw[linet] (2e)--(2e1);
%  \draw[linet] (2e)--(12e);
%  \draw[linet] (e12)--(1e12);
%  \draw[linet] (e12)--(2e12);
%  \draw[linet] (1e2)--(1e12);
%  \draw[linet] (1e2)--(12e2);
%  \draw[linet] (1e1)--(1e12);
%  \draw[linet] (1e1)--(12e1);
%  \draw[linet] (2e2)--(2e12);
%  \draw[linet] (2e2)--(12e2);
%  \draw[linet] (2e1)--(2e12);
%  \draw[linet] (2e1)--(12e1);
%  \draw[linet] (12e)--(12e2);
%  \draw[linet] (12e)--(12e1);
%  \draw[linet] (12e12)--(1e12);
%  \draw[linet] (12e12)--(2e12);
%  \draw[linet] (12e12)--(12e2);
%  \draw[linet] (12e12)--(12e1);
  \draw (e)--(e2);
  \draw (e)--(e1);
  \draw (e)--(1e);
  \draw (e)--(2e);
  \draw (e2)--(e12);
  \draw (e2)--(1e2);
  \draw (e2)--(2e2);
  \draw (e1)--(e12);
  \draw (e1)--(1e1);
  \draw (e1)--(2e1);
  \draw (1e)--(1e2);
  \draw (1e)--(1e1);
  \draw (1e)--(12e);
  \draw (2e)--(2e2);
  \draw (2e)--(2e1);
  \draw (2e)--(12e);
  \draw (e12)--(1e12);
  \draw (e12)--(2e12);
  \draw (1e2)--(1e12);
  \draw (1e2)--(12e2);
  \draw (1e1)--(1e12);
  \draw (1e1)--(12e1);
  \draw (2e2)--(2e12);
  \draw (2e2)--(12e2);
  \draw (2e1)--(2e12);
  \draw (2e1)--(12e1);
  \draw (12e)--(12e2);
  \draw (12e)--(12e1);
  \draw (12e12)--(1e12);
  \draw (12e12)--(2e12);
  \draw (12e12)--(12e2);
  \draw (12e12)--(12e1);
  %Second cell
  \draw[linet] (s1)--(s12);
%  \draw[linet] (s1)--(2s1);
%  \draw[linet] (s12)--(2s12);
%  \draw[linet] (2s1)--(2s12);
  \draw (s1)--(e1);
  \draw (s1)--(1e);
  \draw (s1)--(s12);
  \draw (s1)--(2s1);
  \draw (s12)--(e12);
  \draw (s12)--(1e2);
  \draw (s12)--(2s12);
  \draw (2s1)--(2e1);
  \draw (2s1)--(12e);
  \draw (2s1)--(2s12);
  \draw (2s12)--(2e12);
  \draw (2s12)--(12e2);
  %Third cell
  \draw[linet] (s2)--(s21);
%  \draw[linet] (s2)--(1s2);
%  \draw[linet] (1s2)--(1s21);
%  \draw[linet] (s21)--(1s21);
  \draw (s2)--(e2);
  \draw (s2)--(2e);
  \draw (s2)--(s21);
  \draw (s2)--(1s2);
  \draw (s21)--(e12);
  \draw (s21)--(2e1);
  \draw (s21)--(1s21);
  \draw (1s2)--(1e2);
  \draw (1s2)--(12e);
  \draw (1s2)--(1s21);
  \draw (1s21)--(1e12);
  \draw (1s21)--(12e1);
  %fourth cell
  \draw[linet] (s1s2)--(s1s21);
%  \draw[linet] (s1s2)--(2s1s2);
%  \draw[linet] (2s1s2)--(2s1s21);
%  \draw[linet] (s1s21)--(2s1s21);
  \draw (s1s2)--(s12);
  \draw (s1s2)--(1s2);
  \draw (s1s2)--(s1s21);
  \draw (s1s2)--(2s1s2);
  \draw (s1s21)--(e12);
  \draw (s1s21)--(1s21);
  \draw (s1s21)--(2s1s21);
  \draw (2s1s2)--(12e);
  \draw (2s1s2)--(2s12);
  \draw (2s1s2)--(2s1s21);
  \draw (2s1s21)--(2e12);
  \draw (2s1s21)--(12e1);
  %Fifth cell
  \draw[linet] (s2s1)--(s2s12);
%  \draw[linet] (s2s1)--(1s2s1);
%  \draw[linet] (s2s12)--(1s2s12);
%  \draw[linet] (1s2s1)--(1s2s12);
  \draw (s2s1)--(s21);
  \draw (s2s1)--(2s1);
  \draw (s2s1)--(s2s12);
  \draw (s2s1)--(1s2s1);
  \draw (s2s12)--(e12);
  \draw (s2s12)--(2s12);
  \draw (s2s12)--(1s2s12);
  \draw (1s2s1)--(12e);
  \draw (1s2s1)--(1s21);
  \draw (1s2s1)--(1s2s12);
  \draw (1s2s12)--(1e12);
  \draw (1s2s12)--(12e2);
  %Sixth cell
  \draw (s1s2s1)--(s1s21);
  \draw (s1s2s1)--(2s1s2);
  \draw (s1s2s1)--(s2s12);
  \draw (s1s2s1)--(1s2s1);
\end{tikzpicture}
\]
}
\caption{The partial order on contingency tables $\Xi(3)$ is isomorphic to the two-sided Coxeter complex $\Xi(A_2)$. Highlighted edges indicate the Coxeter complex $\Sigma(A_2)$.}\label{fig:S3}
\end{figure}

Notice that we can permute the balls before insertion, so more than one cover relation can arise from inserting the same bar. For example, using the balls and boxes diagram of Figure \ref{fig:double}, there are two covers that come from inserting a horizontal bar in the gap corresponding to $s_5$:
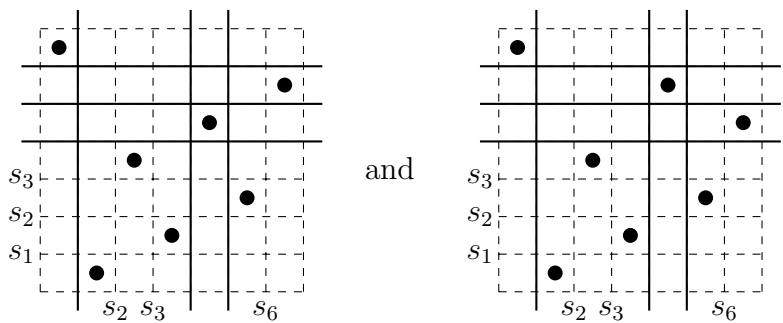
\[
  \begin{tikzpicture}[scale=.25, baseline =1.5cm]
  \draw[dashed] (0,0) grid[step=2] (14,14); 
  \draw (1,13) node[circle,fill=black,inner sep =2] {};
  \draw (3,1) node[circle,fill=black,inner sep =2] {};
  \draw (5,7) node[circle,fill=black,inner sep =2] {};
  \draw (7,3) node[circle,fill=black,inner sep =2] {};
  \draw (9,9) node[circle,fill=black,inner sep =2] {};
  \draw (11,5) node[circle,fill=black,inner sep =2] {};
  \draw (13,11) node[circle,fill=black,inner sep =2] {};
  \draw[thick] (2,-1)--(2,15);
  \draw[thick] (8,-1)--(8,15);
  \draw[thick] (10,-1)--(10,15);
  \draw[thick] (-1,8)--(15,8);
  \draw[thick] (-1,10)--(15,10);
  \draw[thick] (-1,12)--(15,12);
  \draw (4,-1) node {$s_2$};
  \draw (6,-1) node {$s_3$};
  \draw (12,-1) node {$s_6$};
  \draw (-1,2) node {$s_1$};
  \draw (-1,4) node {$s_2$};
  \draw (-1,6) node {$s_3$};
%  \draw (7,-3) node {(right action)};
%  \draw (-3,7) node[rotate=-90] {(left action)};
 \end{tikzpicture}
 \quad\mbox{ and }\quad
 \begin{tikzpicture}[scale=.25,baseline=1.5cm]
  \draw[dashed] (0,0) grid[step=2] (14,14); 
  \draw (1,13) node[circle,fill=black,inner sep =2] {};
  \draw (3,1) node[circle,fill=black,inner sep =2] {};
  \draw (5,7) node[circle,fill=black,inner sep =2] {};
  \draw (7,3) node[circle,fill=black,inner sep =2] {};
  \draw (9,11) node[circle,fill=black,inner sep =2] {};
  \draw (11,5) node[circle,fill=black,inner sep =2] {};
  \draw (13,9) node[circle,fill=black,inner sep =2] {};
  \draw[thick] (2,-1)--(2,15);
  \draw[thick] (8,-1)--(8,15);
  \draw[thick] (10,-1)--(10,15);
  \draw[thick] (-1,8)--(15,8);
  \draw[thick] (-1,10)--(15,10);
  \draw[thick] (-1,12)--(15,12);
  \draw (4,-1) node {$s_2$};
  \draw (6,-1) node {$s_3$};
  \draw (12,-1) node {$s_6$};
  \draw (-1,2) node {$s_1$};
  \draw (-1,4) node {$s_2$};
  \draw (-1,6) node {$s_3$};
%  \draw (7,-3) node {(right action)};
%  \draw (-3,7) node[rotate=-90] {(left action)};
 \end{tikzpicture},
\]
corresponding to
\[
\left[ \begin{array}{rrrr} 1& 0 & 0 & 0 \\ 0 & 0 & 0 & 1\\ 0 & 0 & 1 & 0 \\ 0 & 3 & 0 & 1 \end{array} \right] \quad \mbox{ and } \quad \left[ \begin{array}{rrrr} 1& 0 & 0 & 0 \\ 0 & 0 & 1 & 0\\ 0 & 0 & 0 & 1 \\ 0 & 3 & 0 & 1 \end{array} \right], 
\]
respectively.

Downward covers in the partial order correspond to removing a single bar from the balls in boxes picture, which therefore adds all the entries in two adjacent rows or two adjacent columns of the corresponding contingency tables. In Figure \ref{fig:covers} we see all the upper and lower covers of the table from Figure \ref{fig:double}. The reader might like to translate these arrays into pictures of balls in boxes. In Figure \ref{fig:S3} we see the full refinement order on $\Xi(3)$.

We finish by stating what should be clear at this point.

\begin{prp}
The two-sided Coxeter complex of the symmetric group $S_n$ is isomorphic to $\Xi(n)$ under refinement order.
\end{prp}

\begin{figure}
\[
\begin{tikzpicture}[rotate=-90, xscale=.65,yscale=4]
\coordinate (u1) at (-11,1);
\coordinate (u2) at (-9,1);
\coordinate (u3) at (-7,1);
\coordinate (u4) at (-5,1);
\coordinate (u5) at (-3,1);
\coordinate (u6) at (-1,1);
\coordinate (u7) at (1,1);
\coordinate (u8) at (3,1);
\coordinate (u9) at (5,1);
\coordinate (u10) at (7,1);
\coordinate (u11) at (9,1);
\coordinate (u12) at (11,1);
\coordinate (d1) at (-6,-1);
\coordinate (d2) at (-3,-1);
\coordinate (d3) at (0,-1);
\coordinate (d4) at (3,-1);
\coordinate (d5) at (6,-1);
\coordinate (o) at (0,0);
\draw (o)--(u1);
\draw (o)--(u2);
\draw (o)--(u3);
\draw (o)--(u4);
\draw (o)--(u5);
\draw (o)--(u6);
\draw (o)--(u7);
\draw (o)--(u8);
\draw (o)--(u9);
\draw (o)--(u10);
\draw (o)--(u11);
\draw (o)--(u12);
\draw (o)--(d1);
\draw (o)--(d2);
\draw (o)--(d3);
\draw (o)--(d4);
\draw (o)--(d5);
\draw (o) node[scale=.5,fill=white,inner sep = 0] { $\left[ \begin{array}{ r r r r} 1 & 0 & 0 &0\\ 0 & 0 & 1 & 1 \\ 0 & 3 & 0 & 1 \end{array}\right]$};
\draw (u1) node[scale=.5,fill=white,inner sep = 0, right] { $\left[ \begin{array}{ r r r r} 1 & 0 & 0 &0\\ 0 & 0 & 1 & 0 \\ 0 & 0 & 0 & 1 \\ 0 & 3 & 0 & 1 \end{array}\right]$};
\draw (u2) node[scale=.5,fill=white,inner sep = 0, right] { $\left[ \begin{array}{ r r r r} 1 & 0 & 0 &0\\ 0 & 0 & 0 & 1 \\ 0 & 0 & 1 & 0 \\ 0 & 3 & 0 & 1 \end{array}\right]$};
\draw (u3) node[scale=.5,fill=white,inner sep = 0, right] { $\left[ \begin{array}{ r r r r} 1 & 0 & 0 &0\\ 0 & 0 & 1 & 1 \\ 0 & 0 & 0 & 1 \\ 0 & 3 & 0 & 0 \end{array}\right]$};
\draw (u4) node[scale=.5,fill=white,inner sep = 0, right] { $\left[ \begin{array}{ r r r r} 1 & 0 & 0 &0\\ 0 & 0 & 1 & 1 \\ 0 & 1 & 0 & 0 \\ 0 & 2 & 0 & 1 \end{array}\right]$};
\draw (u5) node[scale=.5,fill=white,inner sep = 0, right] { $\left[ \begin{array}{ r r r r} 1 & 0 & 0 &0\\ 0 & 0 & 1 & 1 \\ 0 & 1 & 0 & 1 \\ 0 & 2 & 0 & 0 \end{array}\right]$};
\draw (u6) node[scale=.5,fill=white,inner sep = 0, right] { $\left[ \begin{array}{ r r r r} 1 & 0 & 0 &0\\ 0 & 0 & 1 & 1 \\ 0 & 2 & 0 & 0 \\ 0 & 1 & 0 & 1 \end{array}\right]$};
\draw (u7) node[scale=.5,fill=white,inner sep = 0, right] { $\left[ \begin{array}{ r r r r} 1 & 0 & 0 &0\\ 0 & 0 & 1 & 1 \\ 0 & 2 & 0 & 1 \\ 0 & 1 & 0 & 0 \end{array}\right]$};
\draw (u8) node[scale=.5,fill=white,inner sep = 0, right] { $\left[ \begin{array}{ r r r r} 1 & 0 & 0 &0\\ 0 & 0 & 1 & 1 \\ 0 & 3 & 0 & 0 \\ 0 & 0 & 0 & 1 \end{array}\right]$};
\draw (u9) node[scale=.5,fill=white,inner sep = 0, right] { $\left[ \begin{array}{ r r r r r} 1 & 0 & 0 & 0 &0\\ 0 & 0 & 0 & 1 & 1 \\ 0 & 1 & 2 & 0 &  1 \end{array}\right]$};
\draw (u10) node[scale=.5,fill=white,inner sep = 0, right] { $\left[ \begin{array}{ r r r r r} 1 & 0 & 0 & 0 &0\\ 0 & 0 & 0 & 1 & 1 \\ 0 & 2 & 1 & 0 &  1 \end{array}\right]$};
\draw (u11) node[scale=.5,fill=white,inner sep = 0, right] { $\left[ \begin{array}{ r r r r r} 1 & 0 & 0 & 0 &0\\ 0 & 0 & 1 & 0 & 1 \\ 0 & 3 & 0 & 1 &  0 \end{array}\right]$};
\draw (u12) node[scale=.5,fill=white,inner sep = 0, right] { $\left[ \begin{array}{ r r r r r} 1 & 0 & 0 & 0 &0\\ 0 & 0 & 1 & 1 & 0 \\ 0 & 3 & 0 & 0 &  1 \end{array}\right]$};
\draw (d1) node[scale=.5,fill=white,inner sep = 0, left] { $\left[ \begin{array}{ r r r r } 1 & 0 & 1& 1\\ 0 & 3 & 0 & 1  \end{array}\right]$};
\draw (d2) node[scale=.5,fill=white,inner sep = 0, left] { $\left[ \begin{array}{ r r r r } 1 & 0 & 0& 0 \\ 0 & 3 & 1 & 2  \end{array}\right]$};
\draw (d3) node[scale=.5,fill=white,inner sep = 0, left] { $\left[ \begin{array}{ r r r  } 1 & 0 & 0 \\ 0 & 1 & 1 \\ 3 & 0 & 1  \end{array}\right]$};
\draw (d4) node[scale=.5,fill=white,inner sep = 0, left] { $\left[ \begin{array}{ r r r  } 1 & 0 & 0 \\ 0 & 1 & 1 \\ 0 & 3 & 1  \end{array}\right]$};
\draw (d5) node[scale=.5,fill=white,inner sep = 0, left] { $\left[ \begin{array}{ r r r  } 1 & 0 & 0 \\ 0 & 0 & 2 \\ 0 & 3 & 1  \end{array}\right]$};
\end{tikzpicture}
\]
\caption{The upper and lower covers of an element of $\Xi(7)$. (The order moves left to right.)}\label{fig:covers}
\end{figure}
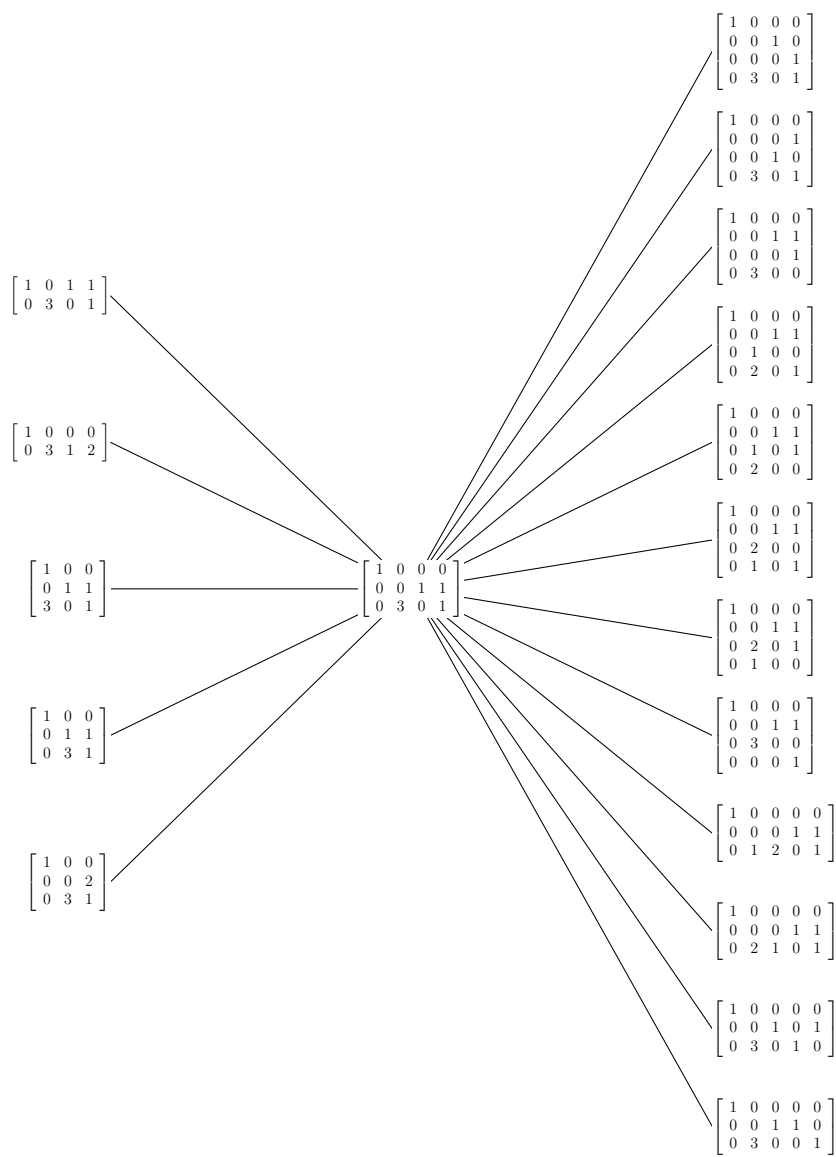

It is well-known that the faces of the Coxeter complex for the symmetric group are modeled by ordered set partitions of $[n]$. Ordered set partitions are in bijection with contingency tables that have $n$ rows (or by those with $n$ columns). To see the correspondence, we simply record, from left to right in each column, the rows that have nonzero entries (counting from bottom to top). For example, the following array corresponds to the ordered set partition $(\{4,5\},\{3,6\},\{1\},\{2\})$:
\[
 \left[\begin{array}{rrrr} 0 & 1 & 0 & 0 \\ 1& 0 & 0 & 0 \\ 1 & 0 & 0 & 0 \\ 0 & 1 & 0 & 0 \\ 0 & 0 & 0 & 1 \\ 0 & 0 & 1 & 0 \end{array} \right].
\]

\begin{rem}
The dual of the type $A_n$ Coxeter complex is the \emph{permutahedron}, which plays an interesting role in the study of combinatorial Hopf algebras, such as the Malvenuto-Reutenauer algebra and the algebra of quasisymmetric functions. See work of Aguiar and Sottile, for example \cite{AS}. 

Suggestively, two-way contingency tables provide an indexing set for a bialgebra known as the set of \emph{matrix quasisymmetric functions}, which contains many well-known combinatorial bialgebras as subalgebras or quotients. See work of Duchamp, Hivert, and Thibon \cite[Section 5]{DHT}. It would be interesting to explore whether $\Xi(n)$ might play a role for the matrix quasisymmetric functions similar to the role the permutahedron plays for the Malvenuto-Reutenauer algebra.
\end{rem}

\begin{rem}
We finish this article by remarking that refinement ordering on contingency tables makes sense not only for two-way tables. A \emph{$k$-way contingency table} of $n$ objects is an array of nonnegative integers
\[
 A = [ a_{i_1,\ldots,i_k} ],
\]
such that the sum of the entries is $n$ and all \emph{marginal sums}
\[
 m_r = \sum_{i_1,\ldots,i_{j-1},i_{j+1},\ldots,i_k} a_{i_1,\ldots,i_{j-1},r,i_{j+1},\ldots,i_k},
\]
are positive. In practical terms, a contingency table involves the study of a population according to several criteria that partition the population, say gender versus age versus income. Requiring the marginal sums to be positive means each criterion is satisfied by at least one member of the population. This seems reasonable, for otherwise the criterion gives no information.

We can inductively define $k$-way contingency tables for $k>2$ by considering $(k-1)$-way tables whose entries are nonnegative integer vectors of the same size, such that when all the vectors with nonzero entries are put into the columns of an array they form a $2$-way table. Refinement order on $k$-way contingency tables whose entries sum to $n$ has maximal elements given by arrays whose marginal sums all equal to $1$. By induction we see there are $(n!)^{k-1}$ maximal tables.

For any $k$, let the set of $k$-way contingency tables whose entries sum to $n$ be denoted by $\Xi(k;n)$. It is not hard to check the partial ordering given by refinement is ranked and boolean, just as in the $2$-way case. (Downward covers are given by adding adjacent entries in some coordinate.) It seems reasonable to expect that we get a shelling order from any linear extension of some sort of natural analogue of two-sided weak order on the facets. If so, refinement ordering on the set of $k$-way contingency tables of $[n]$ defines a thin, shellable simplicial poset and the geometric realization of $\Xi(k;n)$ is a sphere.
\end{rem}

\end{document}